\documentclass[11pt]{article}
\usepackage{authblk} %for authors affiliations
\usepackage{amsmath}
\usepackage{amssymb}
\usepackage{amsfonts}
\usepackage{amsthm}
\usepackage{graphicx}
\usepackage{geometry}
\usepackage[utf8]{inputenc}
\usepackage{enumerate}
\usepackage{enumitem}
\usepackage[english]{babel}
\usepackage{bm}
\usepackage{pdfpages}
\usepackage{float} %per fissare le figure + [H]
\usepackage{latexsym}
\usepackage{soul}
\usepackage{caption}
\usepackage{soul}
\usepackage{bbm}
\usepackage{subcaption}
\usepackage[font=small,labelfont=bf,tableposition=top]{caption}

\definecolor{darkblue}{rgb}{0.0, 0.0, 0.55}
\usepackage[colorlinks=true]{hyperref}
\hypersetup{
	colorlinks=true,
	linkcolor=darkblue,
	filecolor=darkblue,
	citecolor = darkblue,      
	urlcolor=darkblue,
}

\usepackage[round]{natbib} %for separate bibliography %add [numbers] for [..]
\newtheorem{Theorem}{Theorem}[section] 
\newtheorem{Definition}[Theorem]{Definition}
\newtheorem{Assumption}[Theorem]{Assumption}
\newtheorem{Remark}[Theorem]{Remark}

\newtheorem{Proposition}[Theorem]{Proposition}
\numberwithin{equation}{section}

\definecolor{airforceblue}{rgb}{0.36, 0.54, 0.66}

\definecolor{aliceblue}{rgb}{0.94, 0.97, 1.0}

\definecolor{oldlace}{rgb}{0.99, 0.96, 0.9}
\definecolor{burntorange}{rgb}{0.8, 0.33, 0.0}

\definecolor{honeydew}{rgb}{0.94, 1.0, 0.94}
\definecolor{forestgreen}{rgb}{0.13, 0.55, 0.13}

\usepackage[pagewise]{lineno}
\newcommand{\E}[1]{\mathbb{E}\left[#1\right]}
\newcommand{\Prob}[1]{\mathbb{P}\left(#1\right)}

\newcommand{\Y}{\mathbf{Y}}
\newcommand{\Z}{\mathbf{Z}}
\newcommand{\M}{\mathbf{M}}
\newcommand{\kc}{\mathbf{H}}
\newcommand{\y}{\mathbf{y}}
\newcommand{\n}{\mathbf{n}}
\newcommand{\m}{\mathbf{m}}
\newcommand{\vv}{\mathbf{v}}
\newcommand{\e}{\mathbf{e}}
\newcommand{\norm}[1]{\|#1\|_1}
\newcommand{\nth}{^{(n)}}
\newcommand{\minus}{\setminus \{\boldsymbol{0}\}}
\newcommand{\norma}[1]{\left\lVert#1\right\rVert_2}
\newcommand{\norminfty}[1]{\left\lVert#1\right\rVert_\infty}
\newcommand{\metric}{\psi}

\title{\vspace*{-40pt}\bf{Large-sample analysis of cost functionals for inference under the coalescent}}
\date{\today}
\author[, a ]{Martina Favero
\thanks{Corresponding author. \textit{Email address:}  martina.favero@math.su.se}
}
\author[ b , c ]{Jere Koskela}

\affil[, a ]{\small{Department of Mathematics, Stockholm University, 106 91 Sweden }}
\affil[ b ]{School of Mathematics, Statistics and Physics, Newcastle University,  NE1 7RU United Kingdom}
\affil[ c ]{Department of Statistics, University of Warwick, CV4 7AL United Kingdom}

\begin{document}

\maketitle

\begin{abstract}
The coalescent is a foundational model of latent genealogical trees under neutral evolution, but suffers from intractable sampling probabilities.
Methods for approximating these sampling probabilities either introduce bias or fail to scale to large sample sizes.
We show that a class of cost functionals of the coalescent with recurrent mutation and a finite number of alleles converge to tractable processes in the infinite-sample limit.
A particular choice of costs yields insight about importance sampling methods, which are a classical tool for coalescent sampling probability approximation.
These insights reveal that the behaviour of coalescent importance sampling algorithms differs markedly from standard sequential importance samplers, with or without resampling.
We conduct a simulation study to verify that our asymptotics are accurate for algorithms with finite (and moderate) sample sizes.
Our results constitute the first theoretical description of large-sample importance sampling algorithms for the coalescent, provide heuristics for the a priori optimisation of computational effort, and identify settings where resampling is harmful for algorithm performance.
We observe strikingly different behaviour for importance sampling methods under the infinite sites model of mutation, which is regarded as a good and more tractable approximation of finite alleles mutation in most respects.
\end{abstract}

% =============================================================================
%\doublespacing
%\linenumbers

%\vspace{-15pt}
%{\footnotesize \tableofcontents}

\section{Introduction}

%%%%%%%% INTO COALESCENT and LARGE-SAMPLE-SIZE
The coalescent \citep{kingman1982b} is widely used in population genetics, either in its original form or in one of its numerous generalisations, to model or simulate the ancestral history (\textit{genealogy}) of a sample of individuals. 
A crucial quantity for inference under the coalescent is the sampling probability $p(\n)$, i.e.\ the probability of observing a sample $\n\in\mathbb{N}^d\setminus\{\boldsymbol{0}\}$, with $n_i$ being the number of individuals carrying genetic type (\textit{allele}) $i$, and $d$ being the number of possible alleles.
In many statistical applications, the sampling probability constitutes the likelihood function for model parameters or other quantities of interest, such as the mutation rate or effective population size. 
Here we consider a finite number of alleles under recurrent mutation, and neglect other genetic forces such as selection and recombination.
Even in this simple setting the sampling probability is not known explicitly, with the exception of so-called parent-independent mutation discussed in Remark \ref{remark:PIM} below.
A recursive formula for $p(\n)$ is available \citep{lundstrom:1992, sawyer1987}, but intractable in practice when the sample size $\norm{\n}$ is even moderately large.
Our interest is in the large-sample-size regime, to which we give precise meaning in Assumption \ref{assumption:large_sample_size}. 

%%%% SHORT HISTORY OF IS ON THE COALESCENT  
Because of the difficulty of computing the sampling probability exactly, even for moderate sample sizes, Monte Carlo methods have been developed to estimate it.
They broadly split into methods based on tree-valued Markov chain Monte Carlo, importance sampling and sequential Monte Carlo based on simulating coalescent trees sequentially from observed sequences at the leaves to the root, and approximate Bayesian computation which resorts to comparing observed and simulated summary statistics.
Several review articles cover the range of methods available, and we direct the interested reader to \cite{beaumont:2010, marjoram:2006, Stephens2007}.
We will develop an asymptotic description of a class of weighted functionals of the coalescent process, which admits analysis of importance sampling algorithms for large sample sizes as a special case.
To our knowledge, this constitutes the first \textit{a priori} analysis of coalescent importance sampling algorithms.
We begin with an overview of coalescent importance sampling methods.

The history of coalescent inference based on backward-in-time importance sampling starts with the Griffiths--Tavar\'e scheme \citep{griffiths1994simulating}.
%which was only later identified as an importance sampling scheme \citep{felsenstein1999,stephens2000}. 
Subsequently, \cite{stephens2000} developed a more efficient importance sampling algorithm by characterising the family of optimal but intractable proposal distributions, and by defining a tractable approximation.
%This importance sampling scheme was later identified through a more general  diffusion-generator approach by \cite{deiorio2004} and it was generalised from the setting of the finite-allele Kingman coalescent to numerous settings, e.g.\   
Their importance sampling scheme has since been extended in numerous ways, accounting for
the infinite sites mutation model \citep{hobolth2008},
selection \citep{stephens2003},
recombination \citep{fearnhead2001, griffiths2008},
multiple mergers ($\Lambda$-coalescent) \citep{birkner2008, birkner2011, Koskela2015},
and simultaneous multiple mergers ($\Xi$-coalescent) \citep{Koskela2015}.

Importance sampling is closely connected to sequential Monte Carlo, which essentially amounts to generating many sequential importance sampling replicates in parallel and resampling them based on their importance weights to stochastically cull low-weight particles while duplicating high-weight ones \citep[Chapter 9]{chopin:2020}.
Resampling is essential in many applications to prevent exponential growth of importance weight variance in the number of sequential steps \citep{doucet2011}. 
The behaviour of coalescent importance weights is known to differ from this regime \citep{fearnhead2001}, and so-called stopping time resampling has been developed to make resampling applicable to coalescent models \citep{chen2005, jenkins2012stopping}.

It is well known that Monte Carlo methods for the coalescent do not scale well to large sample size or more complex biological models.
As a result, the approximately optimal proposal distributions instigated by \cite{stephens2000} have also been used as probabilistic models in their own right, without importance weighting or rejection control to correct for the fact that they differ from the coalescent sampling distribution.
This approach is particularly prominent in multi-locus settings with recombination \citep{li:2003}.
Indeed, many existing chromosome-scale inference packages rely on these approximate sampling distributions; we mention \texttt{Chromopainter} \citep{lawson:2012} and \texttt{tsinfer} \citep{kelleher2019} as examples.

%%%%% ASYMPT EXPANSION IN LARGE PARAMETERS REGIMES
An entirely different approach to the approximation of the sampling probability consists of deriving series expansions amenable to asymptotics in regimes where some parameters are large. See for example  \cite{jenkins2009,jenkins2010,jenkins2012,jenkins2015}  
%\cite{bks2012, alberti2023} 
for strong recombination, 
\cite{wakeley2008,favero2025,fan2024} for strong selection, 
and \cite{wakeley2009} for strong mutation.
%%%%%%% FOR THE LARGE-SAMPLE-SIZE REGIME, ONLY FIRST ORDER, AND NOT EVEN EXPLICIT
For the large-sample-size regime, the first order of the asymptotic expansion of the sampling probability  is available \citep{favero2022} but it is expressed in terms of the generally unknown stationary density function of an associated Wright--Fisher diffusion.
It does not seem possible to derive a more explicit expression, nor higher orders of the asymptotic expansion, by employing the classical techniques for the large parameters regimes mentioned above. 
% this is because using the recursion formula in the large sample size limit yields trivial expressions (same quantities on both sides, which does not happen with large parameters) 

%%%%%%%%%%%%%%%%%%%%%%%%%%%%%%%%%%%%%%%%%%%%%%%%%%%%%%%
These challenges, together with the canonical nature of the coalescent as a null model of neutral genetic evolution,  motivate our analysis of a class of cost functionals of coalescent block-counting processes for large sample sizes.
A particular choice of costs yields large-sample asymptotics of coalescent importance sampling algorithms as an application. 
We define a large sample size as follows.
%%%% LARGE-SAMPLE-SIZE REGIME:
\begin{Assumption}[Large-sample asymptotics]
\label{assumption:large_sample_size}
We consider samples of the form $n\y_0\nth$, where $\y_0\nth\in\frac{1}{n}\mathbb{N}^d$, 
%and $n\in\mathbb{N}$ becomes large, 
with $\frac{1}{n}\mathbb{N}^d$ being the lattice with spacing $\frac{1}{n}$, and we are interested in large-sample asymptotics, corresponding to $n\to\infty$. 
We assume that 
$$ \lim_{n\to\infty} \y_0\nth = \y_0,$$
for some $\y_0\in \mathbb{R}_+^d$. 
For convenience we also assume $\norm{\y_0}=1$. In this way, the size of the sample $n\y_0\nth$, for large $n$, is approximately equal to $n$. 
\end{Assumption}

%%%%%%%%%%%% CONTRIBUTION OF THIS MANUSCRIPT
In this large-sample regime, we extend a previous convergence result \citep[Thm 2.1]{favero2024} on the block-counting process of the coalescent and the corresponding mutation-counting process to include a sequence of costs.
The convergence of general cost-weighted block-counting processes constitutes one of the two main results of this paper, Theorem \ref{thm:weak_conv}. 
Its proof is an extension and modification of that of \citep[Thm 2.1]{favero2024}, based on the analysis of the tractable parent-independent mutation case, followed by a change of measure between parent-independent and general recurrent mutation. 
This extended convergence result connects the asymptotics of \cite{favero2024} with practical coalescent sequential importance sampling algorithms, which we demonstrate by conducting an a priori analysis of their performance. 
To our knowledge, this work is the first tractable description of these algorithms outside the special case of parent-independent mutation.
Previous work has focused on approximating the optimal but intractable proposal distribution \cite{hobolth2008, stephens2000}, whereas our focus is describing the variance arising from the use of tractable proposal distributions.
The crucial idea for the analysis is based on the following interpretation.
At each step, the discrepancy between a one-step proposal distribution and the intractable true sampling distribution can be viewed as the cost of that step.
We write the sequential importance weights in terms of this cost sequence and employ our convergence result to study the asymptotic behaviour of the weights of classical importance sampling algorithms, particularly those of \cite{griffiths1994simulating} and \cite{stephens2000}.
This constitutes the second main theoretical contribution of the paper, Theorem \ref{thm:weights_convergence}. 

The idea of using a cost framework for the asymptotic analysis of importance sampling algorithms is inspired by the stochastic control approach to rare events simulation. This can be based on large deviations principles when the probability of the rare event is exponentially decaying, e.g.  \cite{dupuis2004}, or it can be based on Lyapunov methods in heavy-tailed settings, e.g. \cite{blanchet2012}. 
These approaches are not applicable to the coalescent, necessitating the development of our bespoke approach based on convergence of cost functionals. 
While the main motivation for the construction of the cost framework is the analysis of importance sampling algorithms, the resulting limit of cost processes is generic and potentially of independent interest.

Our theory makes the surprising prediction that, for large samples, normalised importance weights converge in distribution to 1 under mild conditions which both the \cite{griffiths1994simulating} and \cite{stephens2000} proposal distributions satisfy (c.f.\ Theorem \ref{thm:weights_convergence} and Remark \ref{rmk:general_proposals}).
Such convergence strongly suggests that the only contribution to overall importance weight variance arises from a relatively small number of sequential steps during which the number of remaining lineages in the coalescent tree is small.
This sets the coalescent apart from typical sequential importance sampling applications in which variance of importance weights grows exponentially in the number of steps \citep{doucet2011}.
The fact that the behaviour of coalescent importance weights differs from standard settings has been observed before \citep{fearnhead2008}, but our results predict that the variance of coalescent importance weights remains non-standard even when stopping time resampling is employed.

We conduct a simulation study to show that the predicted pattern of importance weight variance occurs in practice with moderate sample sizes.
We make use of the effect by showing that coalescent sequential importance sampling methods can be improved by using a small number of simulation replicates initially, and branching them out to a large number of replicates once the number of remaining extant lineages becomes small.
The approach of targeting simulation replicates to those sequential steps which contribute to high variance is well-established \citep{lee:2018}, but typically relies on pilot runs to estimate one-step variances.
Our theory facilitates its heuristic use for the coalescent without trial runs.
In a similar vein, we show empirically that resampling, which typically reduces the growth of importance weight variance from exponential to linear in the number of sequential importance sampling steps \citep{doucet2011}, actually reduces the accuracy of the importance sampling algorithm of \cite{stephens2000}, even when stopping time resampling is used.
Thus, the practical contributions of our theory are (i) a heuristic for optimising the allocation of computational resources for finite alleles importance sampling, and (ii) a caution against employing resampling (including stopping time resampling) without a case-by-case assessment of its impact.
The latter is significant because resampling is widely regarded as essential for good performance of practical importance sampling schemes, to which we present a clear and practical counterexample in Section \ref{subsection:fam}.

Finally, while our asymptotic theory is predicated on a finite number of alleles and recurrent mutation, we investigate whether similar empirical results hold for the so-called infinite sites model of mutation (see Section \ref{subsect:ism} for a description).
The infinite sites model is a more tractable approximation of the finite alleles setting, but our results reveal a sharp difference between the two: state-of-the-art infinite sites importance sampling proposal distributions by \cite{stephens2000} and \cite{hobolth2008} exhibit approximately exponential growth of importance weight variance with the number of sequential steps, resampling is effective at reducing Monte Carlo error, and non-uniform allocation of computational resources to different sequential steps does not improve performance.
These results demonstrate that the finite alleles and infinite sites models do not resemble each other when from the perspective of importance sampling methods.
To carry out our infinite sites simulations, we derive some new computational complexity results for the proposal distribution of \cite{hobolth2008} and show that pre-computing an explicit but large matrix reduces its complexity by an order of magnitude.
The matrix in question is independent of observed data and can be reused across all simulations not exceeding a given sample size.

%%%%%% SECTIONS:

The paper is structured as follows. 
In Section \ref{sect:setting} we introduce the coalescent and related sequences, including the cost sequence, and general importance sampling algorithms. 
Section \ref{sect:weak_conv} is dedicated to the convergence of general cost functionals.  
In Section \ref{sect:proposals} we describe and analyse the proposal distributions of specific importance sampling algorithms,
%and by \cite{griffiths1994simulating} and \cite{stephens2000}. 
and, in Section \ref{sect:IS_asympt}, we analyse the asymptotic behaviour of their weights. 
Section \ref{sect:simulations} is dedicated to the simulation study and  
Section \ref{sect:proofs} contains all of the proofs. 
Section \ref{sect:discussion} concludes with a discussion of other applications and future directions of enquiry.

%%%%%%%%%%%%%%%%%%%%%%%%%%%%%%%%%%%%%%%%%%%%%%%%%%%%%%%%%%%%%%%%%%%%%%
%%%%%%%%%%%%%%%%%%%%%%%%%%%%%%%%%%%%%%%%%%%%%%%%%%%%%%%%%%%%%%%%%%%%%%
\section{Setting and notation}
\label{sect:setting}
%%%%%%%%%%%%%%%%%%%%%%%%%%%%%%%%%%%%%%%%%%%%%%%%%%%%%%%%%%%%%%%%%%%%%%
%%%%%%%%%%%%%%%%%%%%%%%%%%%%%%%%%%%%%%%%%%%%%%%%%%%%%%%%%%%%%%%%%%%%%%

\subsection{The coalescent and related sequences of interest}

Given a sample of $n$ individuals, the Kingman coalescent \citep{kingman1982b} models their genealogy backwards in time. 
Starting from the $n$ initial lineages and proceeding backwards in time, each pair of lineages coalesces at rate $1$, and each single lineage undergoes a mutation event at rate $\theta/2>0$. 
We assume there are $d$ possible genetic types, and mutations are sampled from a probability matrix $P=(P_{ij})_{i,j \in \{1, \dots, d\}}$, with  $P_{ij}$ being the forward-in-time probability of a mutation from type $i$ to type $j$.
The matrix $P$ is assumed to be irreducible so as to have a unique stationary distribution.

For a fixed sample size $n \in \mathbb{N}$, we consider the block-counting jump chain $\kc\nth=\{\kc\nth(k)\}_{k\in\mathbb{N}}\subset \mathbb{N}^d\minus$ of the typed version of the coalescent, where $H_i\nth(k)$ is the number of lineages of type $i$ after $k$ jumps in the ancestral history evolving backwards in time, and the coalescent is initialised from a starting configuration of types given by an observed sample $\n\in\mathbb{N}^d\minus$, i.e.\ $\kc\nth(0)=\n$.
The process stops when the most recent common ancestor (MRCA) of all individuals in the sample is reached at step 
\begin{equation*}
    \tau\nth:=\inf\{k\in\mathbb{N}:\norm{\kc\nth(k)}=1 | \kc\nth(0)=\n\}.
\end{equation*}

When not conditioning on $\kc\nth(0)$, the jump chain $\kc\nth$ has a tractable description as a forward-in-time process.
It starts from one ancestor in the past with a type chosen from an initial type distribution, often the stationary distribution of the mutation matrix $P$, and evolves towards the present through mutation and branching events.
The sampling probability $p(\n)$ can be thought of as the probability that this forward process is in state $\n$ at the time of the first branching event which increases its number of lineages to $\|\n\|_1 + 1$.
%The probability that the forward-in-time process reaches a certain configuration $\n$ is the sampling probability $p(\n)$. 
We record the forward and backward transition probabilities of the block-counting jump-chain $\kc\nth$ of the typed Kingman coalescent in Definition \ref{def:forward_transitions} and \ref{def:backward_transitions} below. 
See e.g.\ \cite{stephens2000,deiorio2004} for more details.

\begin{Definition}[Forward transition probabilities]
\label{def:forward_transitions}
The forward-in-time block-counting chain jumps from state $\n\in\mathbb{N}^d\minus$ to the next state $\n+\vv$ with  probability 
    \begin{equation}
    \label{eq:pforward}
    \begin{aligned}
    %\vec{\rho}\nth(\mathbf{v}|\y)
    %=&
    p(\n+\mathbf{v}|\n)
    &=
    \Prob{\kc\nth(k)=\n+\vv|\kc\nth(k+1)=\n} \\
    %\\
    %=&
    &=
    \begin{cases}
    \frac{\norm{\n}-1}{\norm{\n}-1+\theta} \frac{n_j}{\norm{\n}} 
    &\text{  if  }
    \mathbf{v}=\mathbf{\e}_j,
    \quad j\in \{1, \dots, d\},
    \\
    \frac{\theta }{\norm{\n}-1+\theta} \frac{n_i}{\norm{\n}}P_{ij}
    &\text{  if  }
    \mathbf{v}=\mathbf{\e}_j-\mathbf{\e}_i,
    \quad i, j \in \{1, \dots, d\},
    \\
    0 &\text{  otherwise},
    \end{cases}
    \end{aligned}
    \end{equation}
where $\e_i$ is the canonical unit vector with a one at position $i$ and zeros elsewhere.
\end{Definition} 
%%%%%%%%%%%%%%%%%%%%%%%%%%%%%%%%%%%%%%%%%%%%%%%%%%
Note the unnatural indexing of the steps in the forward transition above, going from $k+1$ to $k$. This is chosen intentionally so that the indexing in the following backward transition goes from $k$ to $k+1$. 
In fact, throughout the paper, the indexing follows the backward-in-time direction, which is used more often. 

%%%%%%%%%%%%%%%%%%%%%%%%%%%%%%%%%%%%%%%%%%%%%%%%%%
\begin{Definition}[Backward transition probabilities]
\label{def:backward_transitions}
The backward-in-time block-counting chain jumps from state $\n\in\mathbb{N}^d\setminus \{\boldsymbol{0}\}$ to the next state $\n-\vv$ with  probability 
\begin{equation}
\label{eq:pbackward}
\begin{aligned}
p(\n-\mathbf{v}|\n)
%\rho\nth(\mathbf{v}|\y)
&=
\Prob{\kc\nth(k+1)=\n-\mathbf{v}|
	\kc\nth(k)=\n}
\\
&=
\begin{cases}
\frac{n_j(n_j -1)}{\norm{\n}( \norm{\n} -1 + \theta )} \frac{1}{\pi[j|\n-\e_j]},
&\text{  if  }
\mathbf{v}=\mathbf{\e}_j,
\quad j\in\{1,\dots, d\},
\\
\frac{\theta P_{ij} n_j }
{\norm{\n} (\norm{\n}-1 + \theta ) }
\frac{\pi[i|\n-\e_j]}{\pi[j|\n-\e_j]},
&\text{  if  }
\mathbf{v}=\mathbf{\e}_j-\mathbf{\e}_i,
\quad i, j\in\{1,\dots, d\},
\\
0, &\text{  otherwise},
\end{cases}
\end{aligned}
\end{equation}
where  $\pi[j|\n]$, $j=1, \dots, d$, 
%is a probability distribution over the types which 
can be interpreted as the probability of sampling an individual of type $j$ given that the first  $\norm{\n} $ sampled  individuals have types as in $\n$. In terms of the sampling probabilities, 
\begin{align*}
%\label{defpi}
\pi[i|\n]= \frac{n_i+1}{\norm{\n}+1} \frac{p(\n+\e_i)}{p(\n)}.
%=\frac{k(\n+\e_i)}{k(\n)}
\end{align*}
For $\y\in \frac{1}{n} \mathbb{N}^d\setminus \{\boldsymbol{0}\}$, $n\in\mathbb{N},$ it is also convenient to define 
$$\rho\nth(\mathbf{v}|\y)=p(n\y-\mathbf{v}|n\y).$$
\end{Definition} 
%%%%%%%%%%%%%%%%%%%%%%%%%%%%%%%%%%%%%%%%%%%%%%%%%%

Note the crucial point that the backward transition probabilities are not explicitly known in general since the conditional sampling distribution $\pi[ \cdot | \n] $ is intractable, except for the following special case of parent-independent mutation. 

%%%%%%%%%%%%%%%%%%%%%%%%%%%%%%%%%%%%%%%%%%%%%%%%%%
\begin{Remark}[Parent-independent Mutations (PIM)]
\label{remark:PIM}
Mutations are parent-independent when the type of the mutated offspring does not depend on the type of the parent, i.e.\ $P_{ij}=Q_j,i,j \in \{1,\dots,d\}$. 
In this special case, the sampling probability and the transition probabilities are explicitly known. In particular, 
    $$
    \pi[i|\n]=  
    \frac{ n_i+\theta Q_i}{\norm{\n}+\theta}.
    $$

\end{Remark}
%%%%%%%%%%%%%%%%%%%%%%%%%%%%%%%%%%%%%%%%%%%%%%%%%%
We now briefly define two sequences which are related to the coalescent and will be a useful tool in the rest of the paper. 

\begin{Definition}[Scaled block-counting sequence]
\label{def:Yn}
The sequence of scaled block-counting Markov chains is defined as 
$\Y\nth = \frac{1}{n}\kc\nth\subset \frac{1}{n}\mathbb{N}^{d}, n\in\mathbb{N}$, 
where $n$ represents the sample size which we will take to grow to infinity. 
\end{Definition}

\begin{Definition}[Mutation-counting sequence]
\label{def:Mn}
The sequence of mutation-counting processes is defined as
$\M\nth=(M\nth_{ij})_{i,j=1}^d\subset \mathbb{N}^{d^2}, n\in\mathbb{N},$ 
where $M\nth_{ij}=\{M\nth_{ij}(k)\}_{k\in\mathbb{N}}$, with 
$M\nth_{ij}(k)$ being the cumulative number of mutations from type $i$ to type $j$ (forwards, or $j$ to $i$ backwards) that have occurred in $\Y\nth(0),\dots,\Y\nth(k)$,
i.e. 
    \begin{align*}
    M_{ij}\nth (k)= \sum_{k'=0}^{k-1}
    \mathbb{I}(n\Y\nth(k')-n\Y\nth(k'+1)=\e_j-\e_i),
    \end{align*} 
and $M_{ij}(0)=0$.

\end{Definition}

The asymptotic behaviour of the sequence $(\Y\nth,\M\nth)$, as $n\to\infty,$ was studied by \cite{favero2024}. 
In Theorem \ref{thm:weak_conv} we extend their convergence result to include a sequence $C\nth$ of costs, described in the next subsection, which we will use to analyse importance sampling weights for large sample sizes.

\subsection{The cost sequence and importance sampling}
\label{sect:intro_IS}

Given a sample $n\y_0\nth$, the sampling probability can be written as 
    \begin{align*}
    p(n\y_0\nth)=
    \mathbb{E}_p\left[\mathbb{I}(n\Y\nth(0)=n\y_0\nth)\right]. 
    \end{align*} 
A naive  way to estimate $p(n\y_0\nth)$ is to simulate independent copies of $\Y\nth$  forward in time, following Definition \ref{def:forward_transitions}, and to count how many reach sample size $n + 1$ from configuration $n\y_0\nth$. 
However, as $n$ increases, it becomes rare that a simulation hits $n\y_0\nth$, yielding an estimator with impractically high relative variance. 

The key idea for importance sampling under the coalescent is to simulate backwards, starting from configuration $n\y_0\nth$, according to a proposal distribution $q$, instead of simulating forwards according to the true distribution $p$. 
The change of measure from the forward $p$ to the backward $q$ yields 
    \begin{align*}
    p(n\y_0\nth)=
    \mathbb{E}_q\left[
    L\nth(k)
    \mid 
    n\Y\nth(0)=n\y_0\nth\right],
    \end{align*} 
where
    \begin{align}
    \label{eq:weights}
    L\nth(k)&=
    \frac{p(n\Y\nth(k),\dots,n\Y\nth(0))}{q(n\Y\nth(0),\dots,n\Y\nth(k)\mid n\Y\nth(0)=n\y_0\nth)} \nonumber \\
    &=
    p(n\Y\nth(k)) \prod_{k'=1}^{k} \frac{p(n\Y\nth(k'-1) \mid n\Y\nth(k'))}{q(n\Y\nth(k') \mid n\Y\nth(k'-1))},
    \end{align}
is the importance sampling weight, that is, the likelihood ratio or Radon--Nikodym derivative, of the change of measure. 

Note that the number of sequential steps $k$ in \eqref{eq:weights} is intentionally left general. 
When $k$ is equal to the step $\tau\nth$ at which the MRCA is reached, 
\begin{equation*}
p(n\Y\nth(\tau\nth)) = \sum_{i=1}^d p(\e_i)\mathbb{I}(n\Y\nth(k)=\e_i)
\end{equation*}
is available explicitly and \eqref{eq:weights} corresponds to the importance weight from the importance sampling algorithm with proposal distribution $q$.
Choosing
    $k < \tau\nth$ 
yields truncated algorithms, which are inexact in practice because the factor  $p(n\Y\nth(k))$ is intractable.
For our purposes, truncation is necessary for the asymptotic analysis of importance weights because our scaling limit will only describe algorithms while the number of remaining lineages is large.
In Section \ref{sect:simulations} we show empirically that despite this limitation, our analysis yields insights about algorithms without truncation.
In practice, truncated algorithms have been used for bias-variance trade-off by \cite{jasra2011}, who found that truncating at approx.\ 10\% of the sample size yielded barely perceptible bias for simple mutation models and $n= 100$.
\cite{griffiths1994simulating} also investigated a variance reduction method in which importance sampling was carried out until 2 or 3 lineages remained, and the resulting sampling distribution was computed numerically.
They found truncation at two remaining lineages to be optimal because of the computational cost of numerical methods at higher truncation levels.

The importance sampling estimator is obtained as the average of the importance sampling weights evaluated on independent copies of of $n \Y\nth$, which are simulated  backwards from $n \y_0\nth$ according to the proposal $q$. 
The second moment of this estimator can be written as 
    \begin{align*}
    s(n\y_0\nth)
    =\mathbb{E}_q\left[L\nth(k)^2\mid \Y\nth(0)=\y_0\nth\right]
    =\mathbb{E}_p\left[L\nth(k)\mid \Y\nth(0)=\y_0\nth\right] p(n\y_0\nth).
    \end{align*}
The optimal proposal distribution is the intractable true backward distribution $p$ of Definition \ref{def:backward_transitions}, which yields the zero-variance estimator with  optimal second moment $s(n\y_0\nth)=p(n\y_0\nth)^2$. 
Since optimality cannot be obtained, it is desirable that the estimator is at least asymptotically optimal, which means that it  has bounded relative error, i.e.
    \begin{align*}
    \limsup_{n\to\infty} \frac{s(n\y_0\nth)}{p(n\y_0\nth)^2}
    = \mathbb{E}_p\left[\frac{L\nth(k)}{p(n\y_0\nth)}\mid \Y\nth(0)=\y_0\nth\right]
    < \infty.
    \end{align*}
Therefore, we focus on studying the asymptotic behaviour (under the true distribution) of the normalised importance sampling weights defined as 
    \begin{align}
    \label{eq:normalised_weights}
    W\nth(k)&=  \frac{L\nth(k)}{p(n\y_0\nth)}
    =
    \frac{p(n\Y\nth(k))}{p(n\y_0\nth)} \prod_{k'=1}^{k} \frac{p(n\Y\nth(k'-1) \mid n\Y\nth(k'))}{q(n\Y\nth(k') \mid n\Y\nth(k'-1))}.
    \end{align}
The product in the display above can be thought of as a \textit{cost} as explained in the following. The product goes over the steps of the importance sampling algorithm, at each step the algorithm uses the proposal distribution $q$ in place of the true distribution $p$.  In particular, we can interpret the ratio 
\begin{equation}\label{one_step_cost}
    \frac{p(n\y \mid n\y-\vv)}{q(n\y-\vv \mid n\y)} 
    =: c_q \nth(\vv \mid \y)
\end{equation}
as a one-step cost of using the proposal $q$ in place of the true distribution $p$ in the backward step from $\y\in\frac{1}{n}\mathbb{N}^d\setminus \{\boldsymbol{0}\}$ to $\y-\frac{1}{n}\vv$, for each possible step $\vv \in \{\e_j,\e_j-\e_i\},\; i,j \in \{1,\dots,d\}$. 
Then, the importance sampling weights can be interpreted in terms of the cumulative cost of all the steps. 
More broadly, for a general form of one-step cost, we define the following cost-counting sequence. 

\begin{Definition}[Cost-counting sequence]
\label{def:Cn}
For $n\in\mathbb{N}$, let  $c\nth(\vv\mid\y)$ be a positive function of $\y\in \frac{1}{n}\mathbb{N}^d\setminus \{\boldsymbol{0}\}, \vv \in \{\e_j,\e_j-\e_i\}, \; i, j \in \{1,\dots,d\}$, representing the one-step cost of a backward jump from $\y$ to $\y-\frac{1}{n}\vv$. 
The sequence of cost-counting processes is defined as $C\nth=\{C\nth(k)\}_{k\in\mathbb{N}}\subset \mathbb{R}_+,n\in\mathbb{N}$, where $C\nth(k)$ is the cumulative cost of performing the steps $\Y\nth(0),\dots,$ $\Y\nth(k)$, i.e.  
    \begin{align*}
    C\nth(k)=  \prod_{k'=1}^k c\nth(n\Y\nth(k'-1)-n\Y\nth(k')\mid \Y\nth(k'-1) ) ,
    \end{align*}
and $C\nth(0)=1.$    
\end{Definition}

In the next section, we study the cost $C\nth$ for a general form of one-step cost functions $c\nth$. 
Then, in order to study the asymptotic behaviour of the normalised importance sampling weight $W\nth$ in \eqref{eq:normalised_weights}, we choose $c\nth$ to be of the form $c\nth_q$, as in \eqref{one_step_cost}, for an arbitrary proposal $q$ whose support coincides with $p$.
The description of well-known specific proposals is postponed to Section \ref{sect:proposals}, and the asymptotic analysis of the corresponding costs and weights is in Section \ref{sect:IS_asympt}.

%%%%%%%%%%%%%%%%%%%%%%%%%%%%%%%%%%%%%%%%%%%%%%%%%%%%%%%%%%%%%%%%%%%%%%
\section{Asymptotic analysis of the cost sequence}
\label{sect:weak_conv}
%%%%%%%%%%%%%%%%%%%%%%%%%%%%%%%%%%%%%%%%%%%%%%%%%%%%%%%%%%%%%%%%%%%%%%

Let us recap the initial conditions encountered so far.
\begin{Assumption}[Initial conditions] 
\label{assumption:initial}
Consider the sequence $n\y_0\nth$ of samples of large size satisfying Assumption \ref{assumption:large_sample_size}, and assume $\Y\nth(0)=\y\nth_0$. 
%Let $\y_0\in \mathbb{R}_+^d$, and $\y_0\nth\in\frac{1}{n}\mathbb{N}^d\setminus\{\boldsymbol{0}\}, n\in \mathbb{N} $.
%We assume $\Y\nth(0)= \y_0\nth$, and $ \y_0\nth \to \y_0$ as $n\to\infty$. 
%For the sake of simplicity, and without loss of generality, we also assume %$\norm{\y_0}=1$.
\end{Assumption}
Furthermore, recall that naturally we have $M_{ij}\nth(0)=0,  \forall n\in \mathbb{N}$,  $i,j \in \{1,\dots d\},$ and  $C\nth(0)=1, \forall n\in \mathbb{N}$.
In order to show convergence of the cost sequence of Definition \ref{def:Cn}, we will need the following assumption on the asymptotic behaviour of the cost of one step. 

\begin{Assumption}[Asymptotic cost of one step]
\label{assumption:costs}
There exist some continuous functions $a_j, b_{ij}, \; i,j \in \{1,\dots,d\},$ such that $b_{ij}\geq 1$, 
%continuous in $(E,\psi)$
and
    \begin{align}
    \lim_{n\to \infty} \sup_{\y \in B_\delta\nth }  |n(c\nth(\e_j\mid \y)-1)-a_j(\y)| = 0 
    \quad \text{and} \quad
     \lim_{n\to \infty} \sup_{\y \in B_\delta\nth} |c\nth(\e_j-\e_i\mid\y)-b_{ij}( \y)|=0,
    \end{align}
for each $\delta>0$, where $B_\delta\nth = \{\y \in \frac{1}{n}\mathbb{N}^d: y_j \geq \delta, j \in \{1, \dots, d\}\}$. 
\end{Assumption}
The convergence in Assumption \ref{assumption:costs} is equivalent to uniform convergence on compact sets in the state-space of the technical framework defined in Section \ref{proof:weak_conv}.
It  will be needed for the convergence of the cost sequence, and requires knowledge of the first order approximation of the one-step cost of mutation steps as well as the second order approximation of the one-step cost of coalescence steps.

We can now state the following result, which  extends  \cite[Theorem 2.1]{favero2024}  by including the cost sequence which plays a crucial role in the study of importance sampling algorithms in the next sections. 

\begin{Theorem}[Convergence of general costs]
\label{thm:weak_conv}    
Let 
    $
    \Z\nth=(C\nth, \Y\nth,\M\nth)
    \subset \mathbb{R}_+ \times \frac{1}{n}\mathbb{N}^d\minus \times \mathbb{N}^{d^2},
    n\in\mathbb{N},
    $ 
be the sequence composed by the cost sequence $C\nth$ of Definition \ref{def:Cn}, the scaled block-counting sequence $\Y\nth$ of Definition \ref{def:Yn} evolving backwards in time, and the mutation-counting sequence $\M\nth$ of Definition \ref{def:Mn}, 
with initial conditions given by Assumption \ref{assumption:large_sample_size} and \ref{assumption:initial}.
Assume that the one-step costs satisfy Assumption \ref{assumption:costs}. 
Fix $t\in [0,1)$. 
Then, as $n\to\infty,$ the sequence of processes 
    $ \tilde{\Z}\nth= \{ \Z\nth(\lfloor{sn}\rfloor{} )  \}_{s\in[0,t]}$
converges weakly to the process 
    $ 
    \Z= \{ (C(s), \Y(s ) ,\M(s)) \}_{s\in[0,t]} 
    \subset  \mathbb{R}_+ \times \mathbb{R}_+^d \times \mathbb{N}^{d^2} 
    $, 
defined as follows.  
The state process 
$\Y=\{\Y(s)\}_{s\in[0,t]} = \{(Y_1(s), \ldots, Y_d(s)\}_{s \in [0, t]}$ is the deterministic process defined by
    \begin{align}
    \label{eq:Y}
    \Y(s)= \y_0 \left(1 -s \right);
    \end{align}
the mutation-counting process $\M=(M_{ij})_{i,j=1}^d$ is the matrix-valued process with $M_{ij}=\{M_{ij}(s)\}_{s\in[0,t]}$ being independent time-inhomogeneous Poisson processes with intensities  
    \begin{align*}
     \lambda_{ij}(\Y(s))= \frac{\theta P_{ij} Y_i(s)}{\norm{\Y(s)}^2}
     =
     \frac{\theta P_{ij} y_{0,i}}{(1-s)};
    %= \frac{\theta P_{ij} y_{0,i}}{\norm{\y_0}(\norm{\y_0}-s )}
    \end{align*}
and  the cost process 
$C = \{C(s)\}_{s\in[0,t]}$ is defined by  
    \begin{align}\label{eq:cost}
    C(s) &= \exp\left\{- \int_0^s \langle a(\Y(u)),d\Y(u)\rangle +  \sum_{i,j = 1}^d \int_0^s \log b_{ij}(\Y(u)) dM_{ij}(u)\right\}
    \nonumber \\
    &=
    \exp\left\{\sum_{i=1}^d y_{0,i} \int_0^s  a_i\left(\y_0(1-u)\right)du \right\}
    \prod_{i,j = 1}^d \prod_{k=1} ^{M_{ij}(s)}
     b_{ij}\left(\y_0\left(1-T_{ij}^k\right)\right),
    \end{align}
with 
$T_{ij}^k$ being the time of the $k^{th}$ jump of the process $M_{ij}$. 
\end{Theorem}
\begin{proof}
See Section \ref{proof:weak_conv}.
\end{proof}
Here, converging weakly means converging in the Skorokhod space  $\mathcal{D}_{\mathbb{R}_+^d \times \mathbb{N}^{d^2}\times \mathbb{R}_+} [0,t]$. That is, for  any 
bounded continuous real-valued function $g$ on  $\mathcal{D}_{\mathbb{R}_+^d \times \mathbb{N}^{d^2}\times \mathbb{R}_+} [0,t]$, it yields
    \begin{equation*}
    \lim_{n\to\infty} 
    \E{ g\left(
    \{ \tilde{\Z}\nth(s )  \}_{s\in[0,t]}
    \right)}
    = \E{g\left(
    \{ \Z(s )  \}_{s\in[0,t]}
    \right)}.
    \end{equation*}

\subsection{Heuristic explanation of the convergence}
\label{sect:heuristic_conv}

The one-step transition probabilities of the Markov chain $\Z\nth$, started from a state $(c,\y,\m)\in \mathbb{R}_+ \times \frac{1}{n}\mathbb{N}^d\minus \times \mathbb{N}^{d^2}$, are
\begin{align*}
    &\mathbb{P}(\Z_{k+1}\nth = (c', \y', \m') | \Z_k\nth = (c, \y, \m))\\
    &=
    \begin{cases}
        \rho\nth(\e_j|\y) &\text{if } (c', \y', \m') = (c\,c\nth(\e_j\mid\y), \y -\frac{1}{n}\mathbf{e}_j,\m),\\
        \rho\nth (\mathbf{e}_j-\mathbf{e}_i|\y) &\text{if } (c', \y', \m') = (c\, c\nth (\e_j-\e_i\mid\y), \y -
\frac{1}{n}\mathbf{e}_j+\frac{1}{n}\mathbf{e}_i,\m+\mathbf{e}_{ij}),
    \end{cases}
\end{align*}
where $\rho\nth$ is described in Definition \ref{def:backward_transitions}.
This can be summarised in the following operator $A\nth$, which is the infinitesimal generator of $\tilde{\Z}\nth$, 
    \begin{align}
    \label{eq:An}
    &A\nth f(c,\y, \m) \nonumber \\
    &=
    n \E{f\left(\Z\nth(k+1)\right)-f\left(\Z\nth(k)\right)\mid \Z\nth(k)=(c,\y, \m) }
    \nonumber \\
    & =
    \sum_{j=1}^{d} 
    n\left[f\left(c\, c\nth(\mathbf{e}_j|\y), \y\nth -\frac{1}{n}\mathbf{e}_j,\m\right) -f(c,\y,\m)\right]
    \rho\nth (\mathbf{e}_j|\y)
    \nonumber \\
    &\quad +
    \sum_{i,j=1}^{d}
    \left[f\left(c\, c\nth (\mathbf{e}_j-\mathbf{e}_i|\y), \y -
    \frac{1}{n}\mathbf{e}_j+\frac{1}{n}\mathbf{e}_i,\m+\mathbf{e}_{ij}\right) -f(c,\y,\m)\right]
    n \rho\nth (\mathbf{e}_j-\mathbf{e}_i|\y), 
    \end{align}
where $f$ is a function belonging to a domain to be rigorously determined.
Note the factor $n$ above, which corresponds to scaling time by $n$. 
It is known \citep{favero2022} that, if $\y\nth\to\y\in\mathbb{R}_{+}^d$, then 
    \begin{align}
    \label{eq:lim_trans_prob}
    \rho\nth(\mathbf{e}_j|\y\nth
    )
    \xrightarrow[n\to\infty]{}
    \frac{y_j}{\norm{\y}}, \qquad
    n \rho\nth(
    \mathbf{e}_j-\mathbf{e}_i |\y\nth
    )
    \xrightarrow[n\to\infty]{}
    \lambda_{ij}(\y),
    \quad\quad i,j \in \{1,\dots,d\}.
    \end{align}
Thus, using Assumption \ref{assumption:costs} and first order approximations implies that $A\nth f(c,\y\nth,\m)$ converges to
    \begin{align}
    \label{eq:A}
    A f(c,\y,\m) &=   c\  \partial_c f(c, \y,\m) 
    %\sum_{j=1}^d  \frac{y_j}{\norm{\y}} a_j(\y)
    \left\langle a(\y), \frac{\y}{\norm{\y}}\right\rangle 
    %- \sum_{j=1}^d \frac{y_j}{\norm{\y}}\frac{\partial}{\partial y_j} f(c,\y,\m)
    -\left\langle \nabla_{\y} f(c,\y,\m), \frac{\y}{\norm{\y}} \right\rangle  \nonumber
    \\
    &\quad +
    \sum_{i,j=1}^{d}
    \left[f\left(c\, b_{ij}(\y), \y,\m+\mathbf{e}_{ij}\right) -f(c,\y,\m)\right]
    \lambda_{ij}(\y).
    \end{align}
The operator $A$ above is the infinitesimal generator of the limiting process $\Z=(C,\Y,\M)$ of Theorem \ref{thm:weak_conv}. 
The convergence above is made rigorous in Section \ref{proof:weak_conv}, where it is also proven that this convergence implies Theorem \ref{thm:weak_conv}. 
The crucial tools for the proof are the definition of a proper technical framework, which consists of extending the state space of the processes, and a change-of-measure argument to deal with parent-dependent mutations.  

We now give a brief intuitive explanation of how the limiting process is determined by its infinitesimal generator $A$. 
First, from \eqref{eq:A}, we directly get the following ordinary differential equation for $\Y$:
    \begin{align*}
    &d \Y(s)= - \frac{\Y(s)}{\norm{\Y(s)}} ds,
    \end{align*}
which is trivially solved by \eqref{eq:Y}. 
It is also straightforward to see from \eqref{eq:A} that $M_{ij}$  jumps up by $1$ at rate $\lambda_{ij}(\Y(s))$ independently on the other components of $\M$.   
Finally, for $C$, we get from \eqref{eq:A} the following stochastic differential equation with jumps:
    \begin{align*}
    &d C(t) = C(t) \left\langle a(\Y(t)), \frac{\Y(t)}{\norm{\Y(t)}}\right\rangle  dt 
    + 
     \sum_{i,j=1}^{d} C(t^-) (b_{ij}(\Y(t))-1)dM_{ij}(t).
    \end{align*}
Between jumps, the evolution of $C$ is determined by the drift term, which explains the exponential part of \eqref{eq:cost}. 
The product part of \eqref{eq:cost} is explained by the coefficient $C(t^-)(b_{ij}(\Y(t))-1)$ of $dM_{ij}(t)$ which represents the size of the  jump from $C(t^-)$ to $C(t^-)b_{ij}(\Y(t))$,  given that the mutation-counting process $M_{ij}$ jumps at time $t$. 

%%%%%%%%%%%%%%%%%%%%%%%%%%%%%%%%%%%%%%%%%%%%%%%%%%%%%%%%%%%%%%%%%%%%%%
\section{Proposal distributions}
\label{sect:proposals}
%%%%%%%%%%%%%%%%%%%%%%%%%%%%%%%%%%%%%%%%%%%%%%%%%%%%%%%%%%%%%%%%%%%%%%

In Section \ref{sect:intro_IS} the importance sampling scheme is described in terms of a general backward proposal $q$.
In this section we review two possible choices of $q$ leading to the two well-known importance sampling algorithms by \cite{griffiths1994simulating} and \cite{stephens2000}. Then, we define the corresponding cost of one step and analyse its asymptotic behaviour. 
In the next section, the one-step asymptotic results will be used for the analysis of the corresponding algorithms by using Theorem \ref{thm:weak_conv}. 

\subsection{Griffiths--Tavar\'e (GT) proposal} \label{gt_proposal}

%\cite{griffiths1994simulating} were the first to use a backward proposal distribution, which was only later framed into an importance sampling framework by \cite{felsenstein1999}. 
The \cite{griffiths1994simulating} backward proposal $q_{\scriptscriptstyle GT}$ is proportional to the forward true distribution $p$ of Definition \ref{def:forward_transitions}, that is, 
    \begin{align}
    \label{eq:GT_proposal}
    q_{\scriptscriptstyle GT}(\n-\vv \mid \n) 
    = 
    \frac{p(\n\mid \n -\vv)}{\sum_{\vv'}p(\n\mid \n -\vv')}
    .
    \end{align}
Substituting the proposal $q_{\scriptscriptstyle GT}$ into \eqref{one_step_cost} shows that the cost of a backward step from $\y\in\frac{1}{n}\mathbb{N}^d\setminus \{\boldsymbol{0}\}$  does not depend on the type of step, and, for $\vv \in \{\e_j,\e_j-\e_i\}, \; i,j \in \{1,\dots,d\}$, it is equal to
    \begin{align*}
    c_{\scriptscriptstyle GT}\nth(\vv\mid\y)
    = 
    \frac{p(n\y \mid n \y-\vv)}{q_{\scriptscriptstyle GT}(n\y-\vv \mid n \y)}
    =
    \sum_{\vv'}p(n \y\mid n\y -\vv').
    \end{align*}
Furthermore, for large $n$ we have the following proposition. 
\begin{Proposition}[Asymptotic cost of one GT step]
\label{prop:GT_cost_expansion}
The cost of a backward step from configuration $\y\in\frac{1}{n}\mathbb{N}^d\setminus \{\boldsymbol{0}\}$ in the Griffiths-Tavar\'e algorithm has the following  asymptotic expansion
    \begin{align*}
    c_{\scriptscriptstyle GT}\nth(\vv\mid\y)= 1 - \frac{1}{n} \frac{d-1}{\norm{\y}} + o\left(\frac{1}{n}\right), 
    \quad
    \vv\in\{\e_j,\e_j-\e_i\}, \; i,j\in\{1,\dots,d\}.
    \end{align*}
\end{Proposition}
\begin{proof}
The calculations are reported in Section \ref{proof:GT_cost_expansion}.     
\end{proof}

\subsection{Stephens--Donnelly (SD) proposal}
\cite{stephens2000} derived a proposal of the form
    \begin{align}
    \label{eq:SD_proposal}
    q_{\scriptscriptstyle SD}(\n-\mathbf{v}|\n)
    %\rho\nth(\mathbf{v}|\y)
    =&
    \begin{cases}
    \frac{n_j(n_j -1)}{\norm{\n}( \norm{\n} -1 + \theta )} \frac{1}{\hat{\pi}[j|\n-\e_j]},
    &\text{  if  }
    \mathbf{v}=\mathbf{\e}_j,
    \quad j\in\{1\dots d\},
    \\
    \frac{\theta P_{ij} n_j }
    {\norm{\n} (\norm{\n}-1 + \theta ) }
    \frac{\hat{\pi}[i|\n-\e_j]}{\hat{\pi}[j|\n-\e_j]},
    &\text{  if  }
    \mathbf{v}=\mathbf{\e}_j-\mathbf{\e}_i,
    \quad i, j=\{1,\dots. d\},
    \\
    0, &\text{  otherwise},
    \end{cases}
    \end{align}
where  $\hat{\pi}[j|\n]$, $j=1, \dots, d$, is a family of probability distributions on the space of types. 
In fact, the optimal proposal corresponds to the true backward distribution $p$ of Definition \ref{def:backward_transitions}, which matches the formula above when $\hat{\pi}$ is replaced by $\pi$. 
Since $\pi$ is not known explicitly, except for the case of parent-independent mutation (c.f.\ Remark \ref{remark:PIM}), \cite{stephens2000} propose the following approximation of $\pi$:
%which has efficient properties and can also be characterised  in terms of the infinitesimal generator of the Wright-Fisher diffusion \citep{deiorio2004}: 
    \begin{align*}
    \hat{\pi}[j\mid \n ]
    =
    \sum_{i=1}^d \frac{n_i}{\norm{\n}+\theta}
    \sum_{m=0}^\infty \left(\frac{\theta}{\norm{\n}+\theta}\right)^m(P^m)_{ij},
    \quad j \in \{1,\dots,d\},
    \end{align*}
or equivalently, 
    \begin{align*}
     \hat{\pi}[\cdot|\n]=
     \frac{\n}{\norm{\n}+\theta}
     \left(I-
     \frac{\theta P}{\norm{\n}+\theta}
     \right)^{-1}.
    \end{align*}
Therefore, under the proposal $q_{\scriptscriptstyle SD}$, in the scaled framework, the cost of a backward step from $\y\in\frac{1}{n}\mathbb{N}^d\setminus \{\boldsymbol{0}\}$ to $\y-\frac{1}{n}\vv$ is given by 
    \begin{align*}
    c_{\scriptscriptstyle SD}\nth(\vv\mid\y)
    = 
    \frac{p(n\y \mid n \y-\vv)}{q_{\scriptscriptstyle SD}(n\y-\vv \mid n \y)}
    =
    \begin{cases}
    \hat{\pi}[j|n\y-\e_j] \frac{\norm{\y}}{y_j}
    %\frac{n\norm{\y}-2}{n\norm{\y}-1}\frac{n\norm{\y}-1+\theta}{n\norm{\y}-2+\theta}
    %here I'm taking the -1 instead of the -2, not clear.. but in the end it doesn't matter
    ,
    &\text{  if  }
    \mathbf{v}=\mathbf{\e}_j,
    \quad j \in \{1\dots d\},
    \\
    \frac{\hat{\pi}[j|\n-\e_j]}{\hat{\pi}[i|\n-\e_j]}
    \frac{n y _i -1 + \delta_{ij} }{ny_j },
    &\text{  if  }
    \mathbf{v}=\mathbf{\e}_j-\mathbf{\e}_i,
    \quad i, j \in \{1, \dots, d\},
    \\
    0, &\text{  otherwise}.
    \end{cases}
    \end{align*}
For large $n$ we have the following proposition.
\begin{Proposition}[Asymptotic cost of one SD step]
\label{prop:SD_cost_expansion}
The probability $\hat{\pi}$ of the Stephens-Donnelly proposal distribution has the following  asymptotic expansion
 \begin{align*}
    \hat{\pi}[i\mid n\y-\e_j ]
    =
    \frac{y_i}{\norm{\y}} 
     +
    \frac{1}{n}
    \frac{1}{\norm{\y}}
    \left[
    \frac{y_i(1-\theta)}{\norm{\y}}
     -\delta_{ij}
    +
    \sum_{i'=1}^d  \frac{y_{i'}}{\norm{\y}} \theta P_{i'i}
    \right]
    +o\left(\frac{1}{n}\right),
    \quad
    i,j \in \{1,\dots,d\}.
    \end{align*}
The cost of a backward step from configuration $\y\in\frac{1}{n}\mathbb{N}^d\setminus \{\boldsymbol{0}\}$ in the Stephens-Donnelly algorithm has the following  asymptotic expansion
    \begin{align*}
    &c_{\scriptscriptstyle SD}\nth(\e_j\mid\y)= 1 + \frac{1}{n} \hat{a}_j(\y) + o\left(\frac{1}{n}\right), 
    \quad
    j \in \{1,\dots,d\},
    \end{align*}
where
    \begin{align*}
    \hat{a}_j(\y)
    = \frac{1-\theta}{\norm{\y}}
    -\frac{1}{y_j}\left(1-\sum_{i=1}^d\frac{y_i}{\norm{\y}}\theta P_{ij}\right),
    \end{align*}
and
    \begin{align*}
    &c_{\scriptscriptstyle SD}\nth(\e_i-\e_j\mid\y)= 1 + o(1), \quad i,j \in \{1,\dots,d\}.
    \end{align*}
\end{Proposition}
\begin{proof}
The calculations are reported in Section \ref{proof:SD_cost_expansion}. 
\end{proof}
Note that, in Proposition \ref{prop:SD_cost_expansion}, we only report the first order asymptotic expansion for the cost of a mutation step because that is what we need in the next section in order to apply Theorem \ref{thm:weak_conv}. 

%%%%%%%%%%%%%%%%%%%%%%%%%%%%%%%%%%%%%%%%%%%%%%%%%%%%%%%%%%%%%%%%%%%%%
\section{Asymptotic analysis of importance sampling algorithms}
\label{sect:IS_asympt}
%\subsection{Asymptotic optimality of truncated algorithms }
%%%%%%%%%%%%%%%%%%%%%%%%%%%%%%%%%%%%%%%%%%%%%%%%%%%%%%%%%%%%%%%%%%%%%

Now that we know the asymptotic behaviour of the one-step costs in the GT and SD algorithms, we are able to study the asymptotic behaviour of the corresponding importance sampling weights by employing Theorem \ref{thm:weak_conv}. 

\begin{Remark}[Truncation]\label{remark:truncation}
For each $n\in\mathbb{N}$, we analyse truncated algorithms starting from a sample of the form $n\y_0\nth$ satisfying Assumption \ref{assumption:large_sample_size}, and stopping after $k=\lfloor{tn}\rfloor{}$ steps, for a fixed $t\in(0,1)$.
To get an intuition about the extent of the truncation, consider the following. 
For large $n$, the starting sample size is $n\norm{\y\nth_0}\approx n\norm{\y_0}= n$.
After $\lfloor{tn}\rfloor{}$ steps, the sample size is reduced to approximately $n \norm{\Y\nth(\lfloor{tn}\rfloor{})}$, where $\Y\nth$ follows the proposal distribution. 
The latter is approximately $n \norm{\Y(t)} =n (1-t)$, as justified by Proposition \ref{prop:weak_conv_proposals} below. 
This means that the number of lineages remaining when the truncated algorithm is stopped after $\lfloor t n \rfloor$ steps is approximately $n (1-t)$.
Essentially all practical algorithms are untruncated, corresponding to $t = 1$, but the technical details of our asymptotics only hold for $t < 1$.
In Section \ref{sect:simulations}, we show empirically that our results are robust to this discrepancy, in that none of the simulations we carry out use truncated algorithms.
\end{Remark}

The sequence of Markov chains $\Y\nth$, evolving backwards under the true distribution of Definition \ref{def:backward_transitions}, converges to the deterministic trajectory $\Y$ (see Theorem \ref{thm:weak_conv}, and \citet[Thm 2.1]{favero2024}). 
It is easy to see that the limit remains the same when $\Y\nth$ evolves according to the GT or the SD proposal, after the importance sampling change of measure. This explains the approximation in Remark \ref{remark:truncation} and is more precisely stated in the following proposition for completeness.

\begin{Proposition} \label{prop:weak_conv_proposals}

Let the scaled block-counting sequence of the coalescent $\Y\nth$ evolve  under the GT or SD proposal distribution. That is, $\Y\nth$ is defined as in \ref{def:Yn}, but with backward transition probabilities given by the GT proposal \eqref{eq:GT_proposal} or by the SD proposal \eqref{eq:SD_proposal}, rather than by Definition \ref{def:backward_transitions}. 
Let 
    $
    \Z\nth=(C\nth, \Y\nth,\M\nth)
    %\subset \mathbb{R}_+ \times \frac{1}{n}\mathbb{N}^d\minus \times \mathbb{N}^{d^2},
    %n\in\mathbb{N},
    $ 
be constructed from $\Y\nth$ using Definitions \ref{def:Mn} and \ref{def:Cn}. Then, under Assumptions \ref{assumption:large_sample_size} and \ref{assumption:initial}, the convergence to the limiting process $\Z$ of Theorem \ref{thm:weak_conv} is valid also for the sequence $\Z\nth$ under the GT or the SD proposal distribution.
\end{Proposition}

\begin{proof}
See Section \ref{proof:weak_conv_proposals}. 
\end{proof}

The truncated algorithms are associated to the normalised importance sampling weights defined in \eqref{eq:normalised_weights}, with $k=\lfloor{tn}\rfloor{}$, which can also be written as 
    \begin{align}
    \label{eq:normalised_weights_costs}
    W\nth(k)
    =
    \frac{p(n\Y\nth(k))}{p(n\y_0\nth)} 
    C\nth (k),
    \end{align}
where $C\nth$ is the cost sequence of Definition \ref{def:Cn} with the one-step costs to be chosen to correspond to either the GT or the SD algorithm. 
The asymptotic behaviour of the weights and costs above is analysed in the following.
\begin{Theorem}[Convergence of importance sampling weights]
\label{thm:weights_convergence}
Let $W\nth_{GT}$ and $W\nth_{SD}$ be the normalised importance sampling weights, as defined in \eqref{eq:normalised_weights} or \eqref{eq:normalised_weights_costs}, of the Griffiths--Tavar\'e and the Stephens--Donnelly algorithms respectively.
Let $C\nth_{GT}$ and $C\nth_{SD}$ be the corresponding cost sequences of Definition \ref{def:Cn}. 
Fix $t\in[0,1)$.
Then,
    \begin{align*}
    \frac{p(n\Y\nth(\lfloor{tn}\rfloor{}))}{p(n\y_0\nth)}
    \xrightarrow[n\to\infty]{\mathcal{D}} 
    (1-t)^{1-d};
    \qquad
    C\nth_{GT}(\lfloor{tn}\rfloor{})
    \xrightarrow[n\to\infty]{\mathcal{D}}
    (1-t)^{d-1};
    \qquad  
    C\nth_{SD}(\lfloor{tn}\rfloor{}) 
    \xrightarrow[n\to\infty]{\mathcal{D}}
    (1-t)^{d-1}.
    \end{align*}
Therefore, 
    \begin{align*}
    W\nth_{GT}(\lfloor{tn}\rfloor{})
    \xrightarrow[n\to\infty]{\mathcal{D}}1; 
    \qquad  
    W\nth_{SD}(\lfloor{tn}\rfloor{})
    \xrightarrow[n\to\infty]{\mathcal{D}} 1 ;
    \end{align*}
where $ \xrightarrow[]{\mathcal{D}}$ represents weak convergence, i.e.\ convergence in distribution. 
\end{Theorem}
\begin{proof}
See Section \ref{proof:weights_convergence}.
\end{proof}

Theorem \ref{thm:weights_convergence} shows that very different proposal distributions yield identical importance weights while the sample size remains large.
The performance of the GT and SD schemes is very different in practice \cite[Section 5]{stephens2000}, and Theorem \ref{thm:weights_convergence} does not imply that the performance gap between them will narrow with increasing sample size.
Instead, the interpretation is that the variance of importance weights is dominated by the proposal distribution near the root of the coalescent tree, when then number of remaining lineages is small.
In Section \ref{sect:simulations} we show that this effect is observable in practice with finite sample sizes which are representative of practical data sets.

\begin{Remark}[Convergence conditions for general proposals]\label{rmk:general_proposals}
Consider a general proposal $q^*$ corresponding to the one-step costs $c\nth_*$ of the form \eqref{one_step_cost}, with the following asymptotic expansion 
    \begin{align*}
    &c_{*}\nth(\e_j\mid\y)= 1 + \frac{1}{n} {a}^*_j(\y) + o\left(\frac{1}{n}\right), 
    \\
    &c_{*}\nth(\e_j-\e_i\mid\y)= 1+o(1).
    \end{align*} 
Then a sufficient condition on the second-order coefficients to obtain the convergence result of Theorem \ref{thm:weights_convergence} is
    \begin{align*}
    - \langle \Y(u), a^*(\Y(u))\rangle 
    = d-1.
    %\sum_{i=1}^d Y_i(s) {a}^*_i\left(\Y(s)\right)  = {d-1}.
    \end{align*}
In fact, this condition, together with Theorem \ref{thm:weak_conv}, implies 
    \begin{align*}
    W\nth_{*}(\lfloor{tn}\rfloor{})
    \xrightarrow[n\to\infty]{\mathcal{D}} 
    (1-t)^{1-d}   
    \exp\left\{ \int_0^t \sum_{i=1}^d y_{0,i} \int_0^t  a^*_i\left(\y_0(1-u)\right)du \right\}=1.
    \end{align*}
If the proposal  $q^*$ is of the SD-form, then the corresponding expansion, for the proposed approximation $\pi^*$ of $\pi$, is
    \begin{align*}
    &{\pi}^*[i\mid n\y-\e_j ] = \frac{y_i}{\norm{\y}}+o(1)
    \\&
    {\pi}^*[j\mid n\y-\e_j ] 
    = \frac{y_j}{\norm{\y}} 
    + \frac{1}{n} \tilde{a}^*_{j} (\y)
    +    o\left(\frac{1}{n}\right),
    \text{ with } 
    \tilde{a}^*_{j} (\y) =\frac{y_j}{\norm{\y}} {a}^*_j(\y),
    \end{align*}  
and the sufficient condition corresponds to
    \begin{align*}
    - \sum_{i=1}^d\tilde{a}^*_i(\Y(u)) = \frac{d-1}{\norm{\Y(u)}}.
    \end{align*} 
\end{Remark}

Remark \ref{rmk:general_proposals} gives a criterion on importance sampling proposal distributions under which importance weights do not degenerate with large sample size.
However, by itself the lack of degeneracy indicates very little about the quality of the resulting algorithm in practice.
This is evident because both the GT and SD proposals satisfying \ref{rmk:general_proposals}, yet the latter yields a far superior algorithm.
In our opinion, the practical utility of Theorem \ref{thm:weights_convergence} and Remark \ref{rmk:general_proposals} is in informing the design of particle allocation schedules and resampling schemes for a given proposal distribution, as we show by simulation in Section \ref{sect:simulations} below.

%%%%%%%%%%%%%%%%%%%%%%%%%%%%%%%%%%%%%%%%%%%%%%%%%%%%%%%%%%%%%%%%%%%
\section{Simulation study}
\label{sect:simulations}

\subsection{The finite alleles model}\label{subsection:fam}

To assess the applicability of Theorem \ref{thm:weights_convergence} to finite samples, we carried out a simulation study using the GT and SD proposals.
The code for replicating these simulations is available at \url{https://github.com/JereKoskela/treeIS}.
Runtimes were measured on a single Intel i7-6500U core.

We consider the simulated benchmark data set from Section 7.4 of \cite{griffiths1994simulating}, consisting of 50 samples and 20 sites, with 2 possible alleles at each site.
The true mutation rate is $\theta = 1/2$ per site, and each mutation flips the type of a uniformly chosen site.
We also simulated samples of size 500 and 5000 with the same number of sites and under the same mutation model.
The two larger samples are nested so that the 500 lineages are contained in the 5000 lineage sample, but both are independent of the sample of size 50.
All three samples are provided along with the simulation code.

Figure \ref{fig:cost} shows empirical normalised second moments of importance weights $W_1, \ldots, W_N$ of $N$ simulated replicates in the GT and SD algorithms as a function of the remaining number of lineages.
The empirical normalised second moment is related to the usual effective sample size (ESS) of \cite{kong1994} via
\begin{equation}\label{ess}
    \frac{\frac{1}{N}\sum_{i = 1}^N W_i^2}{\big(\frac{1}{N}\sum_{i = 1}^N W_i \big)^2} = \frac{N}{\text{ESS}}.
\end{equation}
To generate Figure \ref{fig:cost}, independent replicate coalescent trees were initialised from the observed sample, and stopped as soon as they encounter a coalescence event.
Once all replicates had been stopped, the ESS was recorded, simulation of all replicates was restarted, and the cycle of stopping replicates after each coalescence event was iterated until only one lineage remains in each replicate.
To control runtimes, the GT scheme was run using the rejection control mechanism introduced in Section 5.2 of \cite{griffiths1994simulating}, in which realisations with more than a given number of mutations are discarded.
Throughout, we set the discard threshold to 1000.

\begin{figure}[]
\centering
\begin{subfigure}[t]{0.49\textwidth}
\includegraphics[width=\textwidth]{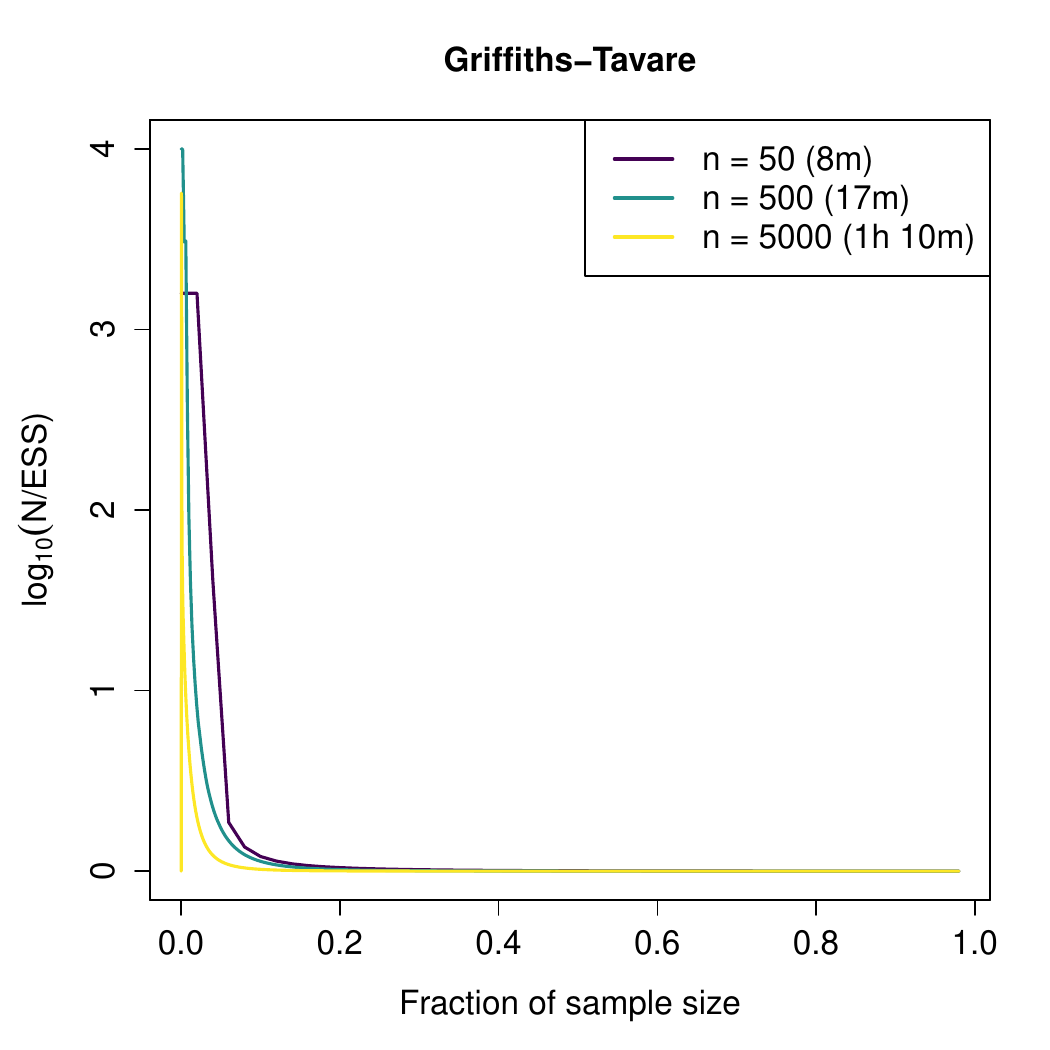}
\caption{}
\label{fig:cost:a}
\end{subfigure}
\begin{subfigure}[t]{0.49\textwidth}
\includegraphics[width=\textwidth]{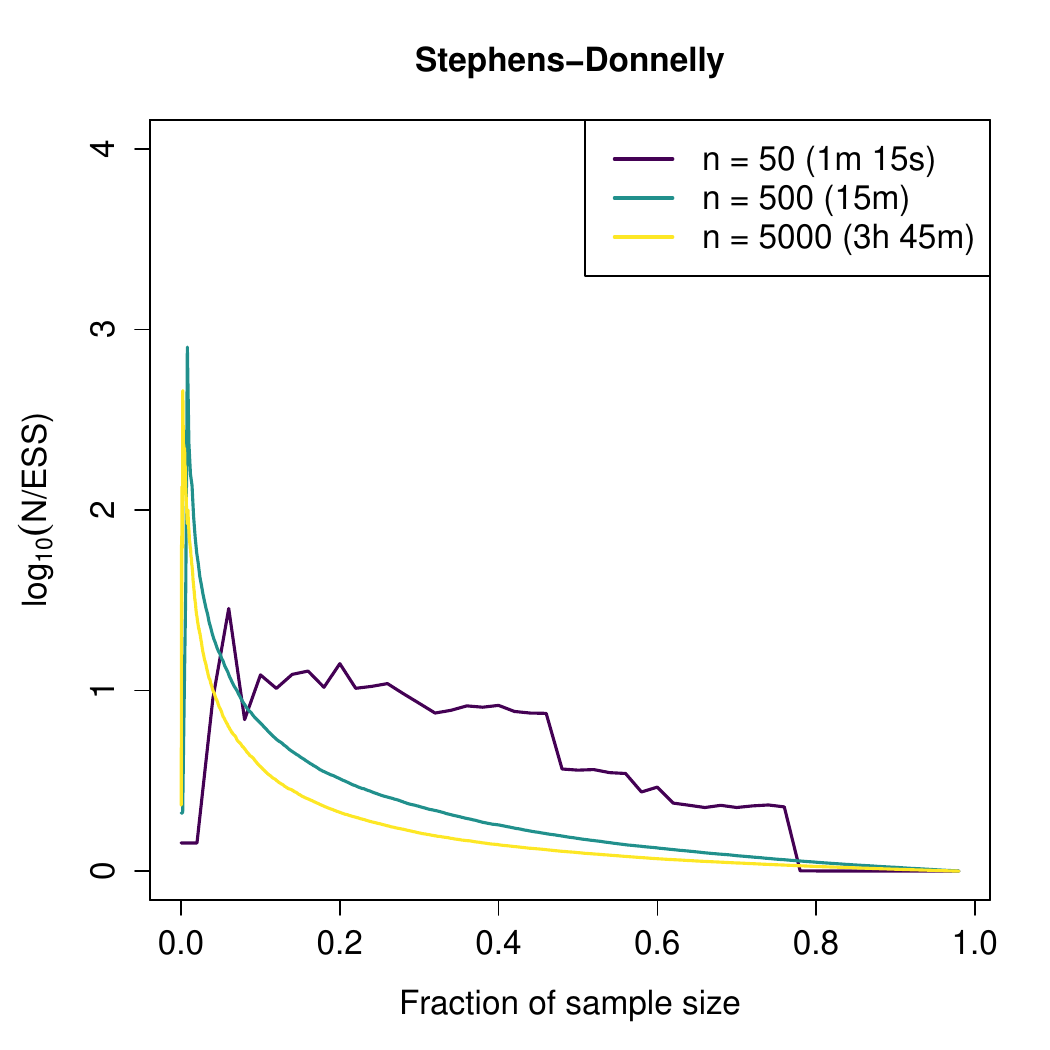}
\caption{}
\label{fig:cost:b}
\end{subfigure}
\caption{Logarithms of normalised second moments of importance weights under the GT  and SD proposals, measured by stopping replicates upon first hitting each fixed number of remaining lineages.
Each figure is an average over 10 000 replicates.}
\label{fig:cost}
\end{figure}

The importance weight variances in both algorithms are plausibly converging towards 0 as $n \to \infty$ except in a neighbourhood of the origin, where they spike sharply.
The convergence is especially rapid for the GT proposal.
Given that the SD proposal is known to be considerably more accurate than GT \cite[Section 5]{stephens2000}, the precise connection between these variance profiles and algorithm performance is not transparent.
However, while it is not an informative indicator of overall algorithm performance, the low variance of weights for large samples evident in Figure \ref{fig:cost} suggests that a small number of replicates could adequately represent the distribution of coalescent trees between the leaves and a low remaining number of lineages near the root.
This would facilitate the allocation of more replicates to the sequential steps close to the root for a given computational budget.
Allocating replicates into steps with high importance weight variance is known to be effective \citep{lee:2018}, but usually requires tuning via trial runs.
Here this optimisation can be carried out \emph{a priori}, at least heuristically.

We tested this idea using the proposal $q_{SD}$ by initialising $\gamma = 100$ independent replicate trees from the configuration of observed leaves $\n$, and simulating each until its number of remaining lineages first hit $\zeta < n$, whose value will be determined below. 
The resulting partially reconstructed trees were sampled with replacement until $\Gamma = 10 000$ replicates were obtained, which were then independently propagated until the root.

Our choice for the value of the threshold $\zeta$ is based on the ansatz that importance weights will begin to vary when the number of lineages has decreased due to coalescence by enough that mutations become commonplace.
Before that point, most proposed steps are coalescences between two lineages which share a type, and the ordering of those events is unlikely to be important.
The standard, untyped coalescent tree with $n$ leaves and mutation rate $\theta / 2$ carries an average of $\theta \log(n)$ mutations when $n$ is large \citep{watterson1975}.
The probability that a given mutation occurs while there are between $\zeta$ and $n$ lineages in the tree is
\begin{equation*}
    \frac{\sum_{j = \zeta}^n j \mathbb{E}[T_j]}{\sum_{j = 2}^n j \mathbb{E}[T_j]} = \frac{\sum_{j = \zeta}^n \frac{1}{j-1}}{\sum_{j = 2}^n \frac{1}{j-1}} \approx \frac{\log(n) - \log(\zeta)}{\log(n)},
\end{equation*}
where $T_j \sim \text{Exp}(\binom{j}{2})$ is the waiting time until the next merger when there are $j$ lineages, and both $\zeta$ and $n$ are large.
Hence, the probability that none of the $\theta \log(n)$ mutations happen before the number of lineages has fallen to $\zeta$ is approximately $( \log(\zeta) / \log(n))^{\theta \log(n)}$.
Equating this to a threshold $\chi \in (0, 1)$ gives
\begin{equation}\label{eq:cutoff}
    \zeta \equiv \zeta(n, \theta) = \lfloor n^{\chi^{1 / (\theta \log (n))}} \rfloor
\end{equation}
as the switch point between $\gamma = 100$ and $\Gamma = 10^4$ replicates.

We simulated sampling probability estimators for a range of mutation rates using the SD proposal, $\chi = 0.1$, and four different importance sampling schedules:
\begin{enumerate}
    \item $\Gamma$ independent replicates of the whole coalescent tree.
    \item $\gamma$ independent replicates of the coalescent tree while it has between $n$ and $\zeta = \lfloor n^{\chi^{1 / (\theta \log (n))}} \rfloor$ lineages, followed by $\Gamma$ replicates as described above.
    \item $\gamma$ independent replicates of the whole coalescent tree.
    \item A number of independent replicates of the whole coalescent tree equal to
    \begin{equation*}
        \frac{\Gamma \zeta(n, \theta) + \gamma (n - \zeta(n, \theta))}{n - 1} \sim \Gamma \chi^{1/\theta} + \gamma (1 - \chi^{1/\theta}).
    \end{equation*}
\end{enumerate}
The rationale for schedule 4 is that it simulates a constant number of replicates across all $n - 1$ coalescence steps while expending approximately the same total computational effort, measured as the total number of simulated coalescence steps, as schedule 2.
We neglect the random number of mutation steps when assessing computational effort because their number is not predictable under the finite alleles model.
However, mutations are rare under the SD proposal, and hence their contribution will be small with very high probability.

\begin{figure}[]
	\centering
	\includegraphics[width=0.49\textwidth]{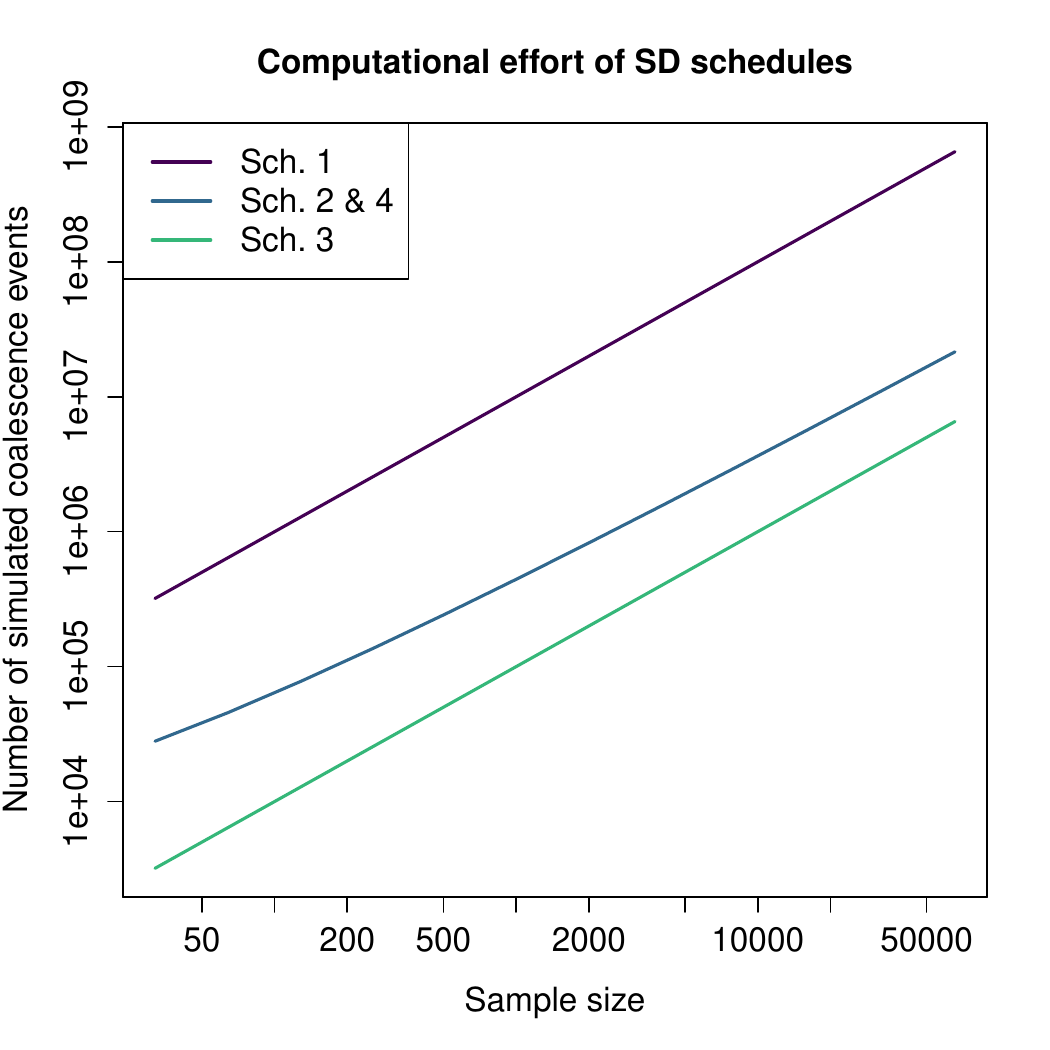}
	\caption{Number of simulated coalescence steps from the one-step proposal distribution $q_{SD}( \cdot | \cdot )$ for the four schedules with $\Gamma = 10^4$, $\gamma = 100$, $\theta = 0.5$, and $\chi = 0.1$.
		Note the log-scale on both axes.}
	\label{fig:cost-profiles}
\end{figure}

The approximate computational effort of executing all four schedules are depicted in Figure \ref{fig:cost-profiles}.
The purpose of these schedules is to assess how the number of simulated replicates interacts with the sample size $n$ and the presence or absence of resampling, and whether or not the non-uniform allocation of computational effort in schedule 2 is beneficial.
Both practical questions are motivated by Theorem \ref{thm:weak_conv}, which suggests that importance sampling methods might stabilise for large sample sizes even for a fixed number of particles, and that consequently, it is advantageous to expend more computational effort near the root of the tree where the number of remaining lineages is low.
This is in stark contrast to general importance sampling theory, which dictates that the number of replicates needs to grow exponentially in the dimension of the latent object \citep[Theorem 1.1]{chatterjee2018sample}.
Here, the dimension of the latent ancestral tree is approximately linear in sample size.

\begin{figure}[H]
\centering
\begin{subfigure}[t]{0.49\textwidth}
\includegraphics[width=\textwidth]{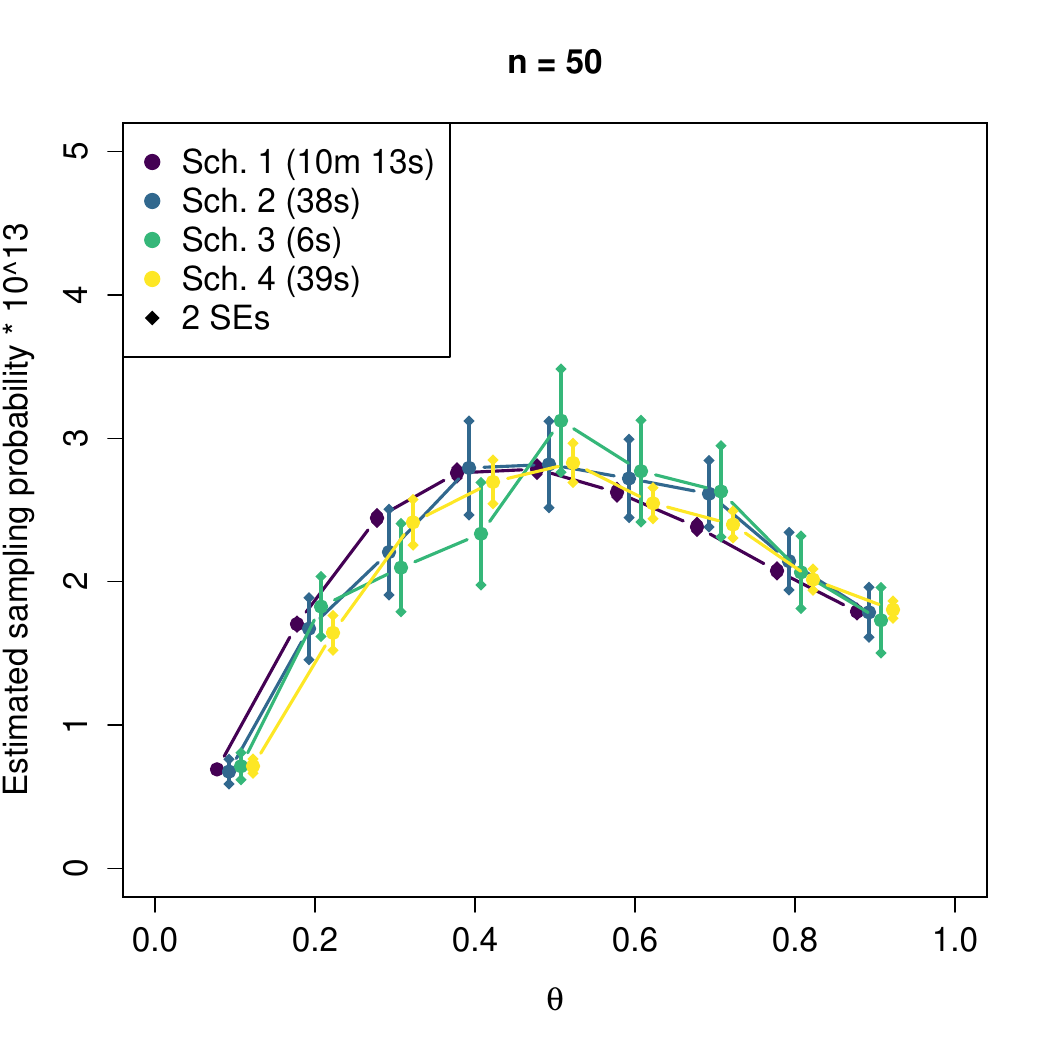}
\caption{}
\label{fig:sd-surfaces:a}
\end{subfigure}
\begin{subfigure}[t]{0.49\textwidth}
\includegraphics[width=\textwidth]{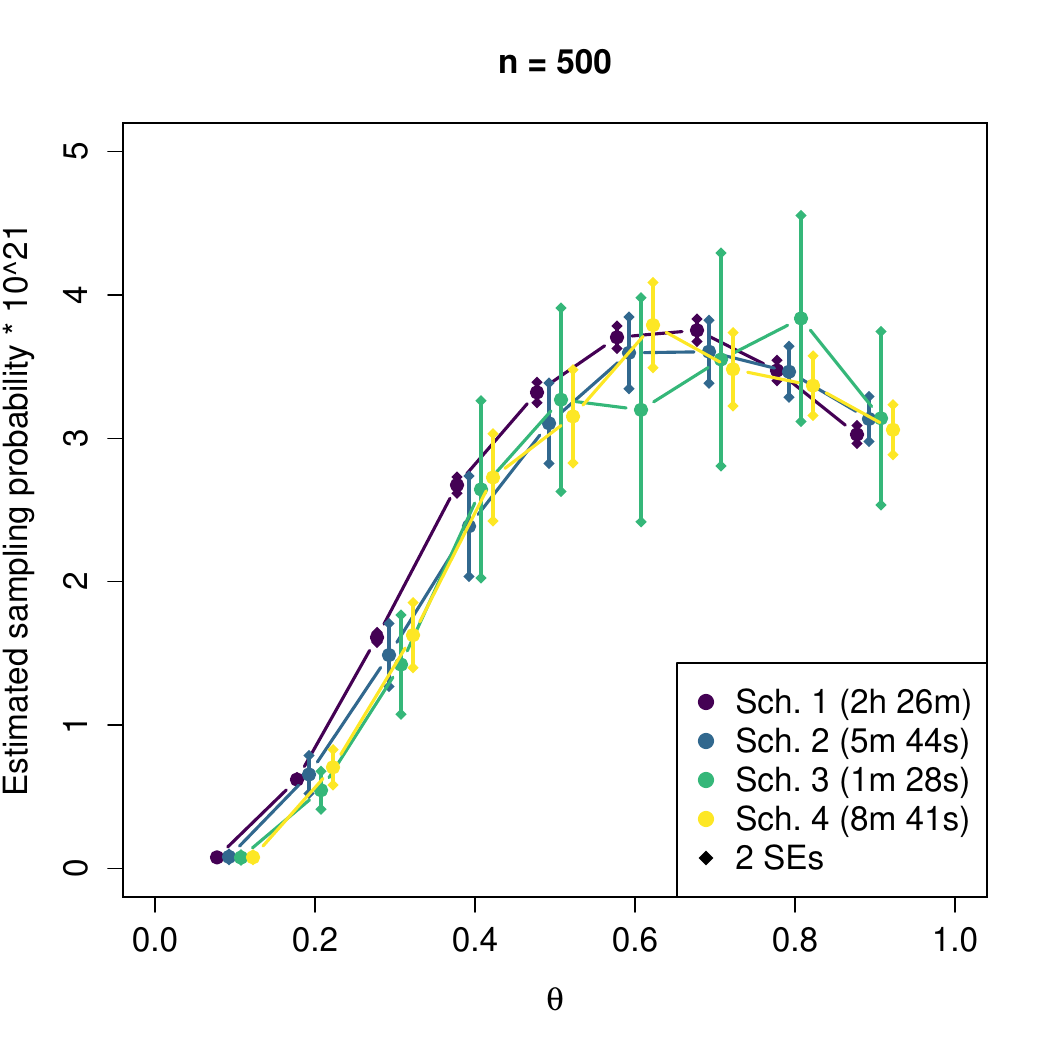}
\caption{}
\label{fig:sd-surfaces:b}
\end{subfigure}
\begin{subfigure}[t]{0.49\textwidth}
\includegraphics[width=\textwidth]{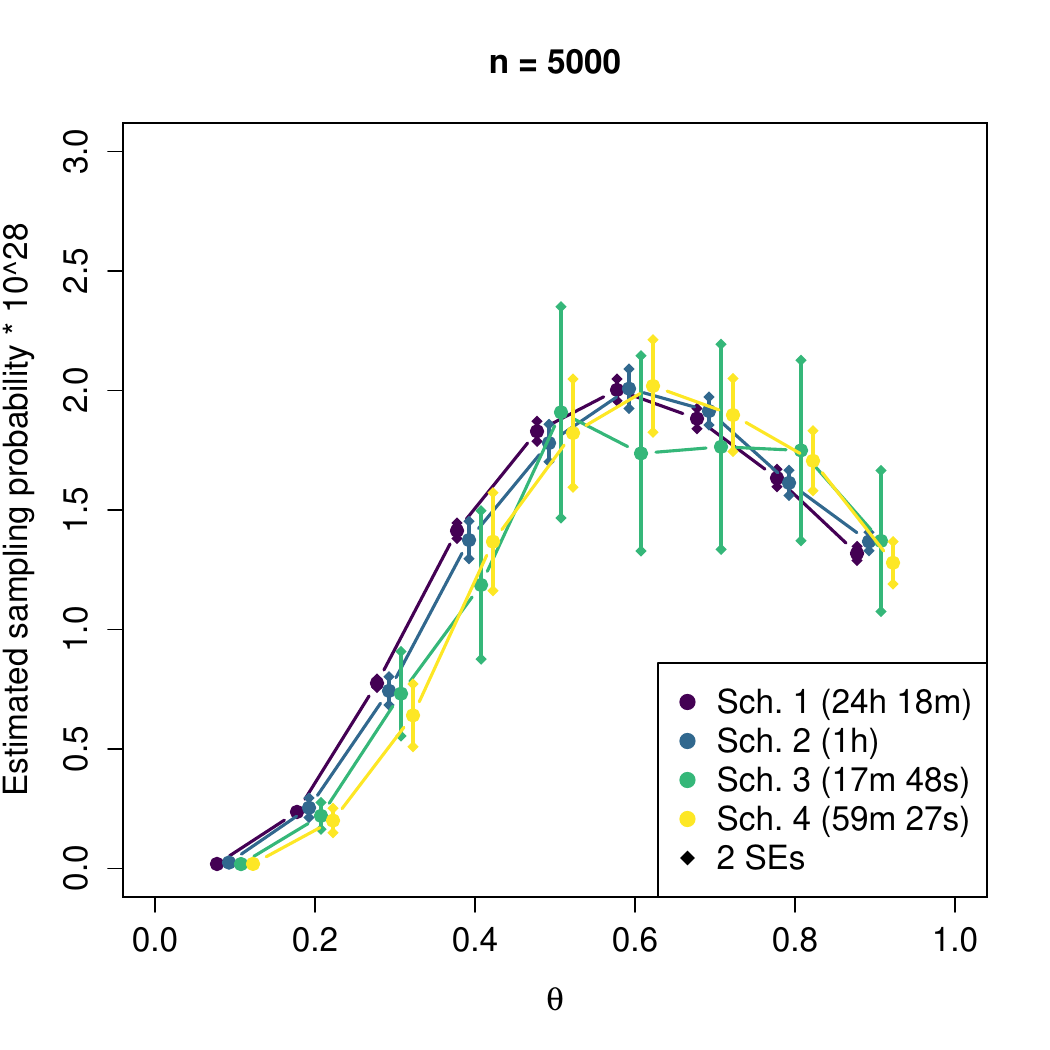}
\caption{}
\label{fig:sd-surfaces:c}
\end{subfigure}
\caption{Performance of the four schedules with $\gamma = 10^2$ and $\Gamma = 10^4$ for various sample sizes based on independent simulations at points $\theta \in \{0.1, 0.2, \ldots, 0.9\}$. 
Standard errors were computed using the method of \cite{chan:2013} for schedule 2, where replicates are not independent.
The data-generating parameter is $\theta = 0.5$.
Each y-axis is multiplied by appropriate, large constant to aid visualisation, and small horizontal offsets have been artificially added to all four curves in each panel for visual clarity.}
\label{fig:sd-surfaces}
\end{figure}

Figure \ref{fig:sd-surfaces} shows simulated sampling distribution estimators for three different sample sizes using all four schedules and a range of values of the mutation rate $\theta$.
Good performance of a schedule is indicated by a smooth graph with small standard errors.
In line with the prediction of Theorem \ref{thm:weak_conv}, there is no noticeable reduction in performance for larger sample sizes: $n = 50$ and $n = 5000$ show comparable performance with the same number of simulated replicates.
It is also clear that the $10^4$ replicates of schedule 1 are needlessly expensive for accurate sampling probability estimation.
Schedule 3 with 100 replicates is by far the fastest but somewhat noisy.
Schedule 2 is also much faster than schedule 1, and nearly as accurate.
Notably, it is both faster and slightly more accurate than schedule 4, so that the allocation of more replicates near the root at the cost of fewer replicates elsewhere is delivering a boost in accuracy.
For $n = 5000$, the most challenging case, schedules 1 and 2 are virtually indistinguishable but the latter is faster by a factor of 24, while schedules 3 and 4 have noticeably larger standard errors.

\begin{figure}[]
	\centering
	\includegraphics[width=0.49\textwidth]{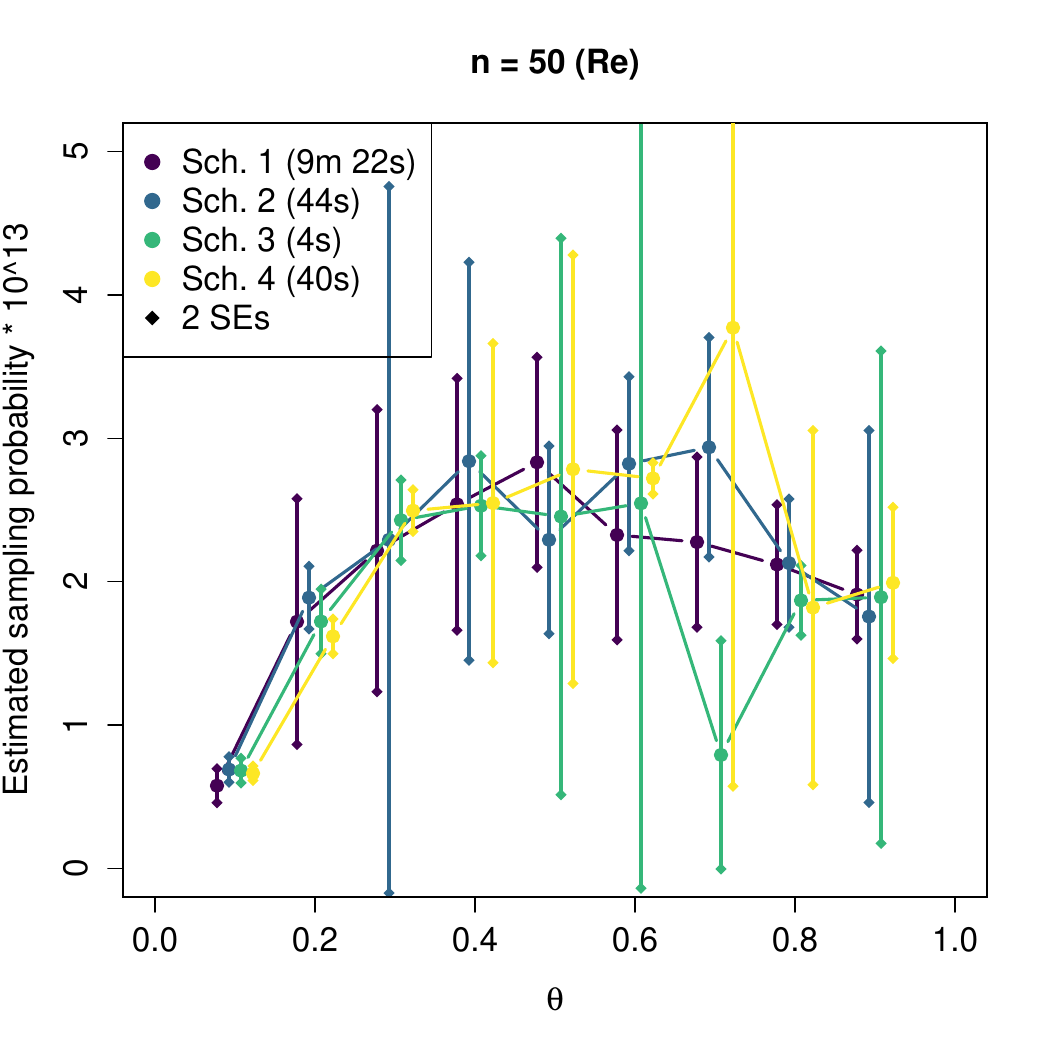}
	\caption{A repeat of the simulation in Figure \ref{fig:sd-surfaces:a} in which replicates were stopped whenever the number of lineages decreased.
		Once all replicates had stopped, systematic resampling \citep[Section 9.6]{chopin:2020} was performed if the effective sample size, ESS in \eqref{ess}, was less than 10\% of the number of replicates.
		The y-axis is expressed in units of $10^{-13}$ to aid visualisation, and small horizontal offsets have been artificially added to all four curves for visual clarity.}
	\label{fig:resampling}
\end{figure}

So far we have focused on importance sampling without resampling.
Figure \ref{fig:cost} suggests that importance weights at intermediate times are poor predictors of their final values, and invites the question of whether resampling based on importance weights is beneficial.
It is well-known that, for the coalescent, resampling partially constructed replicates after a fixed number of simulation steps is harmful \citep{fearnhead2008}.
The standard remedy is so-called stopping-time resampling, in which partially reconstructed trees are stopped when the number of remaining lineages hits a given level, and resampling is performed once all replicates have been stopped \citep{chen2005, jenkins2012stopping}.
This schedule of resampling is an exact parallel of the method of stopping replicate simulations for representative importance weight variance calculation described above Figure \ref{fig:cost}.
Figure \ref{fig:resampling} makes clear that, for the standard coalescent and the SD proposal, resampling at these stopping times can also be harmful: sampling probability estimates are noisier and have larger standard errors than those in Figure \ref{fig:sd-surfaces:a}, which depicts the same simulation without resampling.
Schedules 3 and 4 in particular depict substantial Monte Carlo error, which results in incorrect estimates of both the mean and standard error.
It is noteworthy that the method depicted in Figure \ref{fig:resampling} used a restrictive resampling threshold, $\text{ESS} < N / 10$, which typically only triggered 2--5 times per run for $N = 10^4$ particles in this example.
Thus, even a small amount of resampling had a marked, harmful effect on estimator performance.
Typical guidance on resampling for generic particle filtering and SMC is to resample when $\text{ESS} < N / 2$ \cite[Section 10.2]{chopin:2020}. 
In this case the change in threshold would make little difference since, as shown in Figure \ref{fig:cost}, importance weight variance is low for most of the run before spiking rapidly at the end.
For a less accurate proposal distribution such as GT, stopping time resampling does dramatically improve inference, and it has also been observed to be beneficial for the SD proposal under a microsatellite model of mutation \cite[Section 6]{chen2005}.

\subsection{The infinite sites model}\label{subsect:ism}

The infinite sites model (ISM) is a more analytically and computationally tractable approximation of the site-by-site description of the finite alleles model.
The genome of a lineage is associated with the unit interval $[0, 1]$, which is also taken to be the type of the MRCA.
Mutations occur along the branches of the coalescent tree with rate $\theta / 2$, and each mutation is assigned to a uniformly sampled location along the genome.
Mutations are inherited along the tree towards the leaves, so that the type of a sampled leaf is the list of mutations which occur on the branches connecting it to the MRCA.
The list of mutations carried by an individual is referred to as its \emph{haplotype}.
The infinite sites approximation prohibits the same position mutating more than once, and is a good approximation when mutations are rare and the number of sites is large.

It is convenient to describe a sample of individuals from the infinite sites model as a triple $(\mathbf{S}, \mathbf{n}, \bm{\ell})$, where $\mathbf{S}$ is a matrix which lists observed haplotypes in its rows, with multiplicities given by $\mathbf{n}$, and where the location of each mutant site is listed in $\bm{\ell}$.
If $h \leq n$ distinct haplotypes composed from a total of $r$ mutations are observed in a sample of  $n$ individuals, then $\mathbf{S}$ is an $h \times r$ matrix with $S_{i, j} = 1$ if haplotype $i$ carries mutation $j$, and 0 otherwise.
The corresponding entry $n_i$ is the number of times haplotype $S_i = ( S_{i,1}, \ldots, S_{i,r} )$ was observed, and $\ell_j \in [0, 1]$ is the genomic location of the $j$th mutation.

The forward transition density under the ISM are very similar to the transition probabilities in the finite alleles case:
\begin{align*}
    &p(\mathbf{S}', \n', \bm{\ell}'|\mathbf{S}, \n, \bm{\ell})\\
    &=
    \begin{cases}
    \frac{\norm{\n}-1}{\norm{\n}-1+\theta} \frac{n_i}{\norm{\n}} 
    &\text{  if  }
    (\mathbf{S}', \mathbf{n}', \bm{\ell}') =(\mathbf{S}, \mathbf{n} + \mathbf{e}_i, \bm{\ell}),
    \quad i \in \{1,\dots, h\},
    \\
    \frac{\theta }{\norm{\n}-1+\theta} \frac{n_i}{\norm{\n}}
    &\text{  if  }
    (\mathbf{S}', \mathbf{n}', \bm{\ell}')=(E_{ij} \mathbf{S}, a_j(\mathbf{n}, 1), a_j(\bm{\ell}, x)),
    \quad i \in \{1,\dots h\}, \; j \in \{0, \ldots, r\},
    \\
    0 &\text{  otherwise},
    \end{cases}
\end{align*}
where $a_j(\mathbf{v}, x)$ is the vector obtained from $\mathbf{v}$ by inserting the scalar $x$ between then $j$th and $(j+1)$th positions, and $E_{ij}$ is an operator which inserts a duplicate of row $i$ as the new last row of $\mathbf{S}$, and then inserts $\mathbf{e}_{h + 1}$ as a new column in the $j$th position.
The backward transition probabilities are intractable, similarly to the finite alleles case, and don't depend on the labels $\bm{\ell}$ so we suppress them from the notation going forward, for the sake of readability.

There are three backward-in-time IS proposal distributions available for the ISM: one due to \cite{griffiths1994ancestral} (GT), an approximation of the optimal proposal due to \cite{stephens2000} (SD), and an improved approximation by \cite{hobolth2008} (HUW).
To describe them, it will be convenient to borrow notation from \cite{song2006} and introduce the set $\mathcal{M} \equiv \mathcal{M}(\mathbf{S}, \mathbf{n}) \subset \{ 1, \ldots, h \}$ of row indices which bear at least one mutation present only in that row, and for which the corresponding entry of $\n$ is 1.
Such a mutation is called a \emph{singleton}.
For $j \in \mathcal{M}$, we write $S_j^{\omega}$ for the row obtained from $S_j$ by flipping the singleton $S_{j, \omega}$ from 1 to 0.
For a mutation $\omega \in \{1, \ldots r\}$, let $d_{\omega} := \sum_{i = 1}^h S_{i, \omega} n_i$ be the number of samples on which it appears.
Then, the three proposal distributions are
\begin{align*}
    q_{GT}(\mathbf{S}', \n'|\mathbf{S}, \n) &\propto
    \begin{cases}
        (n_j - 1) &\text{if } (\mathbf{S}', \mathbf{n}') = (\mathbf{S}, \mathbf{n} - \mathbf{e}_j) \text{ and } n_j \geq 2, \\
        \theta (n_{j'} + 1) / \| \n \|_1 &\text{if } n_j = 1, j \in \mathcal{M}, \text{ and } \exists \omega \; \& \; j' \neq j : (S')_{j'} = S_j^{\omega},\\
        \theta / \| \n \|_1 &\text{if } n_j = 1, j \in \mathcal{M}, \text{ and } \exists \omega : (S')_j = S_j^{\omega},\\
        0 &\text{otherwise},
    \end{cases}\\
    q_{SD}(\mathbf{S}', \n'|\mathbf{S}, \n) &\propto
    \begin{cases}
        n_j &\text{if } (\mathbf{S}', \mathbf{n}') = (\mathbf{S}, \mathbf{n} - \mathbf{e}_j) \text{ and } n_j \geq 2, \\
        1 &\text{if } n_j = 1, j \in \mathcal{M} \text{ and } \exists \omega \; \& \; j' : (S')_{j'} = S_j^{\omega},\\
        0 &\text{otherwise},
    \end{cases}\\
    q_{HUW}(\mathbf{S}', \n'|\mathbf{S}, \n) &\propto
        \sum_{\omega = 1}^r u_{j,\omega}(\theta),
\end{align*}
where
\begin{align*}
    &u_{j,\omega}(\theta) := \\
    &\begin{cases}        
    \displaystyle \frac{n_j}{d_{\omega}} \frac{\displaystyle\sum_{k = 2}^{\norm{\n} - d_{\omega} + 1} \frac{d - 1}{(\norm{\n} - k)(k - 1 + \theta)} \binom{\norm{\n} - d_{\omega} - 1 }{k - 2} \binom{\norm{\n} - 1}{k - 1}^{-1}}{\displaystyle\sum_{k = 2}^{\norm{\n} - d_{\omega} + 1} \frac{1}{k - 1 + \theta} \binom{\norm{\n} - d_{\omega} - 1}{k - 2} \binom{\norm{\n} - 1}{k - 1}^{-1}} &\text{if } S_{j, \omega} = 1, \\
    \displaystyle \frac{n_j}{\norm{\n} - d_{\omega}}\left(1 - \frac{\displaystyle\sum_{k = 2}^{\norm{\n} - d_{\omega} + 1} \frac{d - 1}{(\norm{\n} - k)(k - 1 + \theta)} \binom{\norm{\n} - d_{\omega} - 1 }{k - 2} \binom{\norm{\n} - 1}{k - 1}^{-1}}{\displaystyle\sum_{k = 2}^{\norm{\n} - d_{\omega} + 1} \frac{1}{k - 1 + \theta} \binom{\norm{\n} - d_{\omega} - 1}{k - 2} \binom{\norm{\n} - 1}{k - 1}^{-1}}\right) &\text{if } S_{j, \omega} = 0,
    \end{cases}
\end{align*}
and where the support of $q_{HUW}$ is all states $(\mathbf{S}', \mathbf{n})$ which are reachable from $(\mathbf{S}, \mathbf{n})$ by coalescing two identical lineages or removing one singleton mutation.
The HUW proposal also requires special treatment for some edge cases, such as two remaining lineages separated by $k_1$ and $k_2$ mutations; see \cite[Section 3.2]{hobolth2008} for details.

The complexity of evaluating $u_{j, \omega}(\theta)$ is linear in the number of lineages $\norm{\n}$.
Hence the complexity of evaluating $q_{HUW}$ is $O(\norm{\n} r)$.
Sampling a step from $q_{HUW}$ requires evaluating it for all $h$ haplotypes, and sampling one coalescent tree requires $\norm{\n} - 1 + r$ steps.
Thus, the overall complexity per replicate is $O(\norm{\n} r h (\norm{\n} + r))$, or 
\begin{equation*}
    O(\norm{\n}^2 \theta^2 (\log \norm{\n})^2 + \norm{\n} \theta^3 (\log \norm{\n})^3)
\end{equation*}
using the asymptotics $r \sim h \sim \theta \log(\norm{\n})$ which hold for the coalescent in expectation \citep{watterson1975}.
This cost is prohibitive both for large samples $\norm{\n}$, and for large sequence lengths, with which $\theta$ grows linearly.

To render the HUW proposal practical, note that for a fixed value of $\theta$ the large sums in the numerator and denominator required to evaluate $u_{j, \omega}(\theta)$ can be pre-computed for all required values of $\norm{\n}$ between 2 and the number of observed lineages, and all possible values of $d_{\omega} \in \{1, \ldots, \norm{\n} - 1\}$.
The resulting matrix requires $O(\norm{\n}^2)$ storage, but is independent of the observed data. 
With this matrix in place, $u_{j, \omega}(\theta)$ can be evaluated in $O(1)$ time.
Moreover, the whole proposal distribution $q_{HUW}( \cdot, \cdot | \mathbf{S}, \mathbf{n})$ can be computed once for a given sample size, and only needs to be recomputed after a coalescence event, at which point it requires a re-traversal of the whole matrix $\mathbf{S}$.
A simulation step which removes a mutation affects only the row and column of $\mathbf{S}$ in which that mutation features, requiring only an $O(r + h)$ update rather than a full $O(r h)$ re-computation of the proposal distribution.
As a result, the computational complexity reduces to three components:
\begin{enumerate}
    \item $\norm{\n} - 1 + r$ steps, each of which requires a sample from $q_{HUW}( \cdot, \cdot | \mathbf{S}, \mathbf{n})$ at $O(h)$ cost per step,
    \item $\norm{\n} - 1$ computations of $q_{HUW}( \cdot, \cdot | \mathbf{S}, \mathbf{n})$ at cost $O(r h)$ each,
    \item and $r$ partial refreshes of $q_{HUW}( \cdot, \cdot | \mathbf{S}, \mathbf{n})$ at cost $O(r + h)$ per step.
\end{enumerate}
With the expected growth of $r$ and $h$ with $\norm{n}$ under the coalescent, the total cost per replicate tree is
\begin{align}
    O((\norm{\n} + r) h + \norm{\n} r h + r(r + h)) &= O(\norm{\n} \theta \log \norm{\n} + 3 \theta^2 (\log \norm{\n})^2 + \norm{\n} \theta^2 (\log \norm{\n})^2) \notag \\
    &= O(\norm{\n} \theta^2 (\log \norm{\n})^2),\label{eq:huw_cost}
\end{align}
improving the scaling in both sample size and sequence length by a linear factor.
However, the infinite sites SD proposal is substantially faster at a cost of $O(h)$ per step, or
\begin{equation}\label{eq:sd_cost}
    O((\norm{\n} + r) h) = O(\norm{\n} \theta \log \norm{\n} + \theta^2 (\log \norm{\n})^2)
\end{equation}
per replicate tree.

Theorem \ref{thm:weak_conv} does not apply to the ISM.
However, since the ISM is regarded as a good approximation to the finite alleles model for long sequences and rare mutations, it is instructive to examine whether similar conclusions about importance sampling proposal distributions hold.
To that end, we applied all three ISM proposal distributions to the data set of \cite{ward1991}---a common benchmark with $n = 55$ samples and $r = 18$ mutations.
To assess scaling, we also simulated two synthetic data sets with respective sizes $n = 550$ and $n = 5500$ using $\theta = 5.0$, which is the approximate maximum likelihood estimator from the \cite{ward1991} data set.
For HUW, we set the driving value of $\theta$ used to pre-compute the proposals for each data set equal to the Watterson estimator \citep{watterson1975}, which takes respective values 3.93, 4.94, and 4.90 for the three data sets.
The largest matrix took around 2 hours of computing time in serial, but the computation is trivial to parallelise and can be reused for any data set with size no greater than 5500 and for which 4.9 is an acceptable driving value for the population-rescaled mutation rate.

\begin{figure}[]
\centering
\begin{subfigure}[t]{0.49\textwidth}
\includegraphics[width=\textwidth]{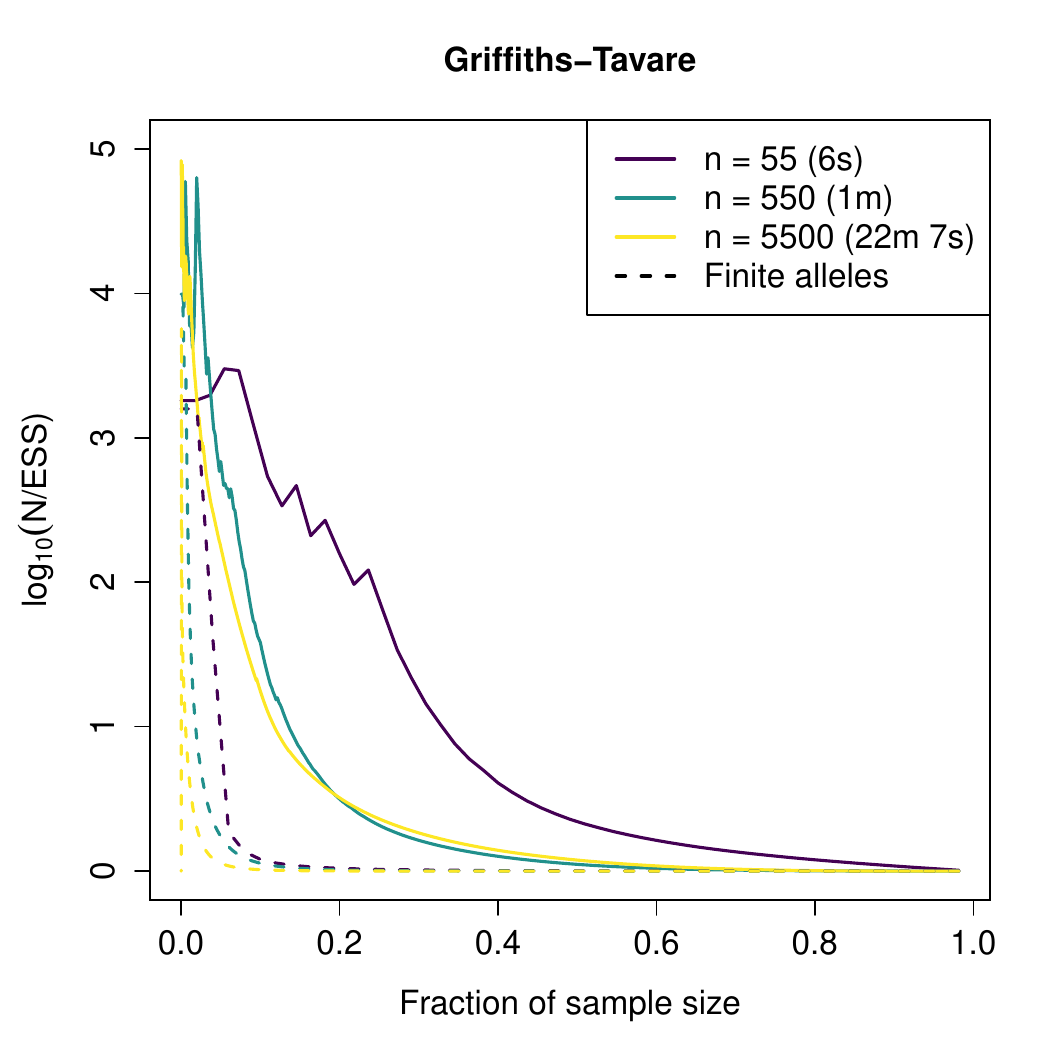}
\caption{}
\label{fig:cost-ism:a}
\end{subfigure}
\begin{subfigure}[t]{0.49\textwidth}
\includegraphics[width=\textwidth]{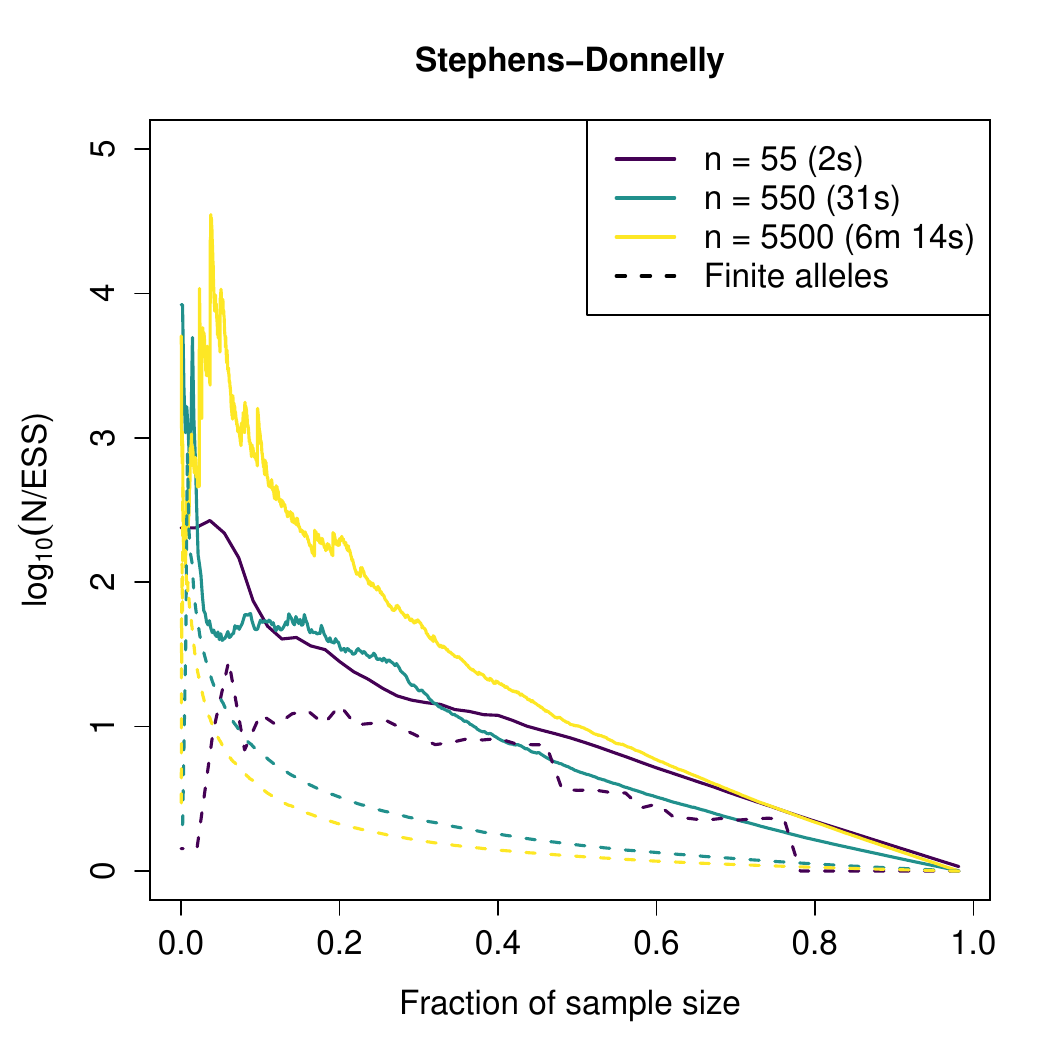}
\caption{}
\label{fig:cost-ism:b}
\end{subfigure}
\begin{subfigure}[t]{0.49\textwidth}
\includegraphics[width=\textwidth]{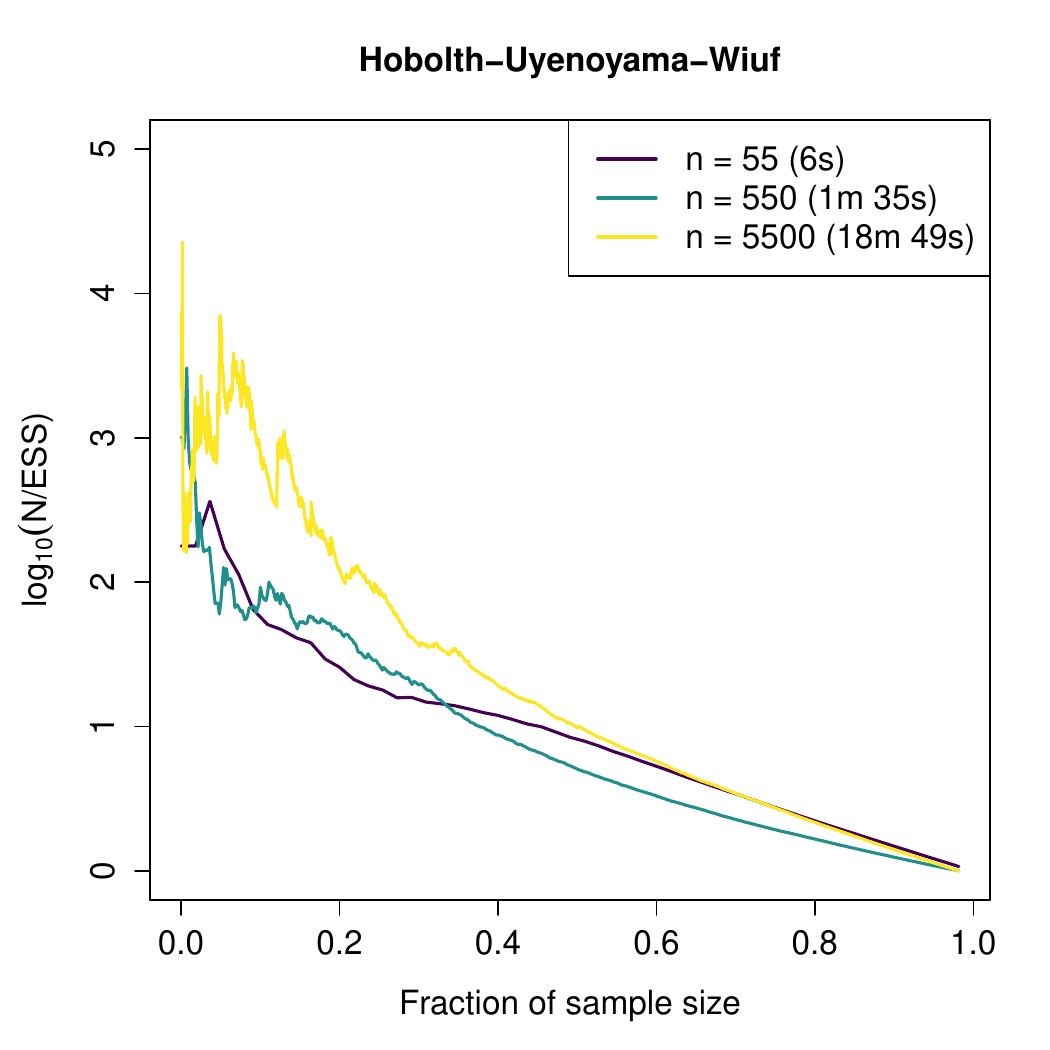}
\caption{}
\label{fig:cost-ism:c}
\end{subfigure}
\caption{Logarithms of normalised second moments of importance weights for the GT, SD, and HUW proposals, measured by stopping replicates upon first hitting each fixed number of remaining lineages.
Each figure was obtained by averaging $10^5$ replicates.
The results from Figure \ref{fig:cost:a} and \ref{fig:cost:b} are reproduced in dashed lines in panels \ref{fig:cost-ism:a} and \ref{fig:cost-ism:b} for ease of comparison.}
\label{fig:cost-ism}
\end{figure}

Figure \ref{fig:cost-ism} repeats the analysis from Figure \ref{fig:cost} for the ISM and the three proposals.
While the GT proposal appears consistent with Figure \ref{fig:cost}, albeit with slower convergence, the behaviour of the variances under the more practical SD and HUW proposals are qualitatively different.
Indeed, they are close to straight lines (on a log-scale), in line with the usual exponential growth of importance weight variance in the absence of resampling \citep{doucet2011}.
The fact that variances increase throughout the simulation run suggests i) that there may be no particular benefit in allocating more particles near the end of the simulation, and ii) that resampling will be effective.

We tested these suggestions by simulating sampling distribution estimators independently for a range of values of $\theta$, using the four replicate schedules from Section \ref{subsection:fam}.
Figure \ref{fig:likelihood-ism-55} bears out both suggestions for the data set with $n = 55$ samples: the results with resampling are considerably less noisy than those without, except for schedule 3 with only 1000 particles which has very high standard error.
There is also very little difference between schedules 1, 2, and 4.
Schedule 3 becomes erratic when resampling is applied, showing substantial Monte Carlo error in estimates of both the mean and standard error.
Based on non-overlapping confidence intervals, schedules 1, 2, and 4 also appear to underestimate standard errors slightly when resampling is used.

\begin{figure}[H]
\centering
\begin{subfigure}[t]{0.49\textwidth}
\includegraphics[width=\textwidth]{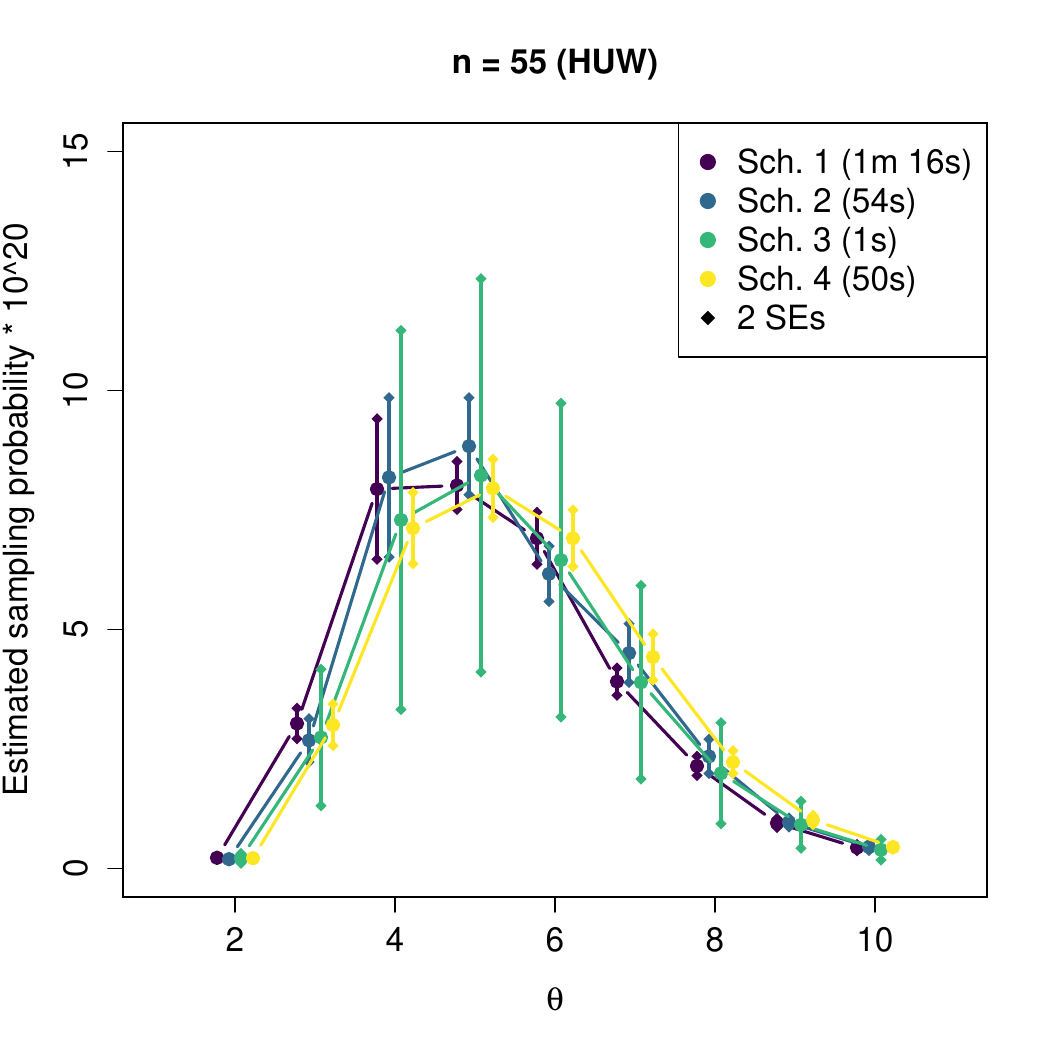}
\caption{}
\label{fig:likelihood-ism-55:a}
\end{subfigure}
\begin{subfigure}[t]{0.49\textwidth}
\includegraphics[width=\textwidth]{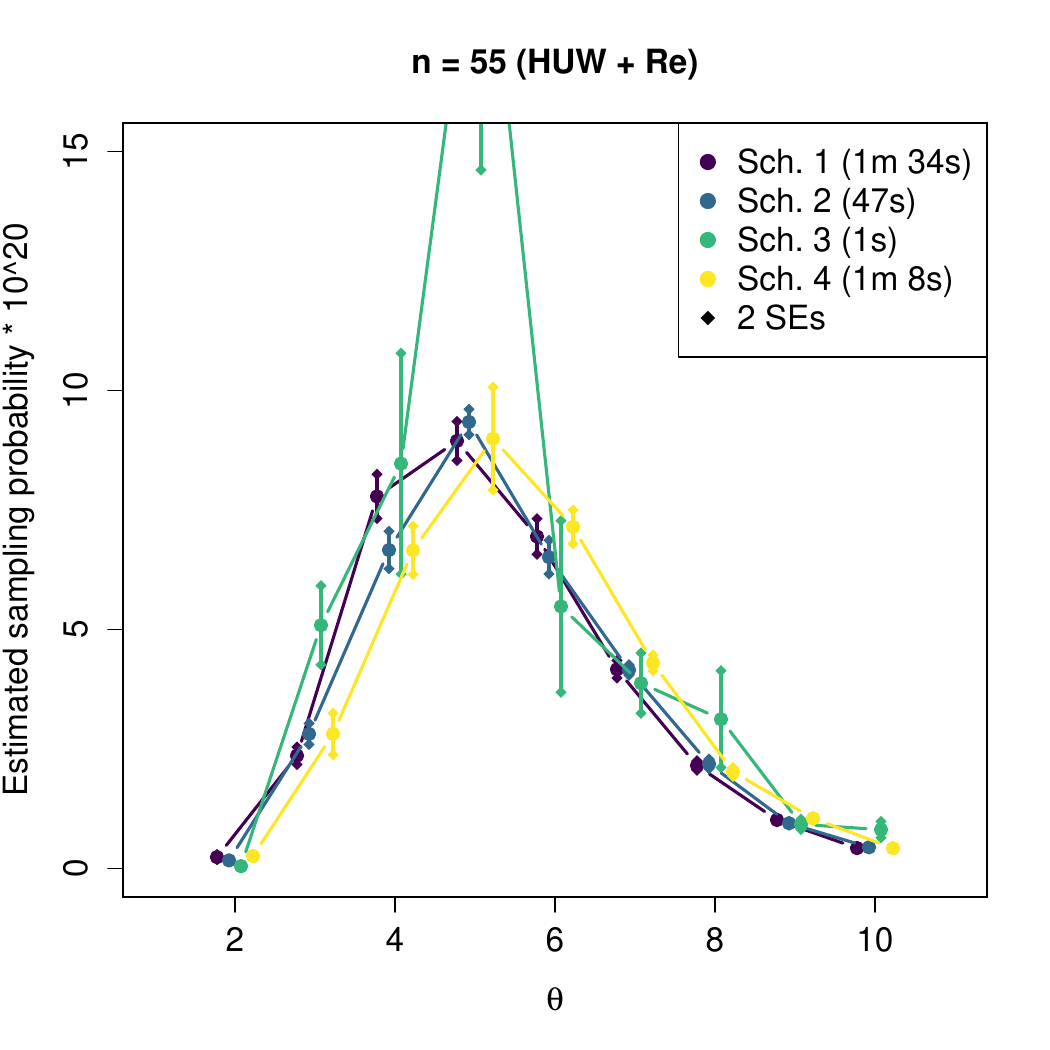}
\caption{}
\label{fig:likelihood-ism-55:b}
\end{subfigure}
\begin{subfigure}[t]{0.49\textwidth}
\includegraphics[width=\textwidth]{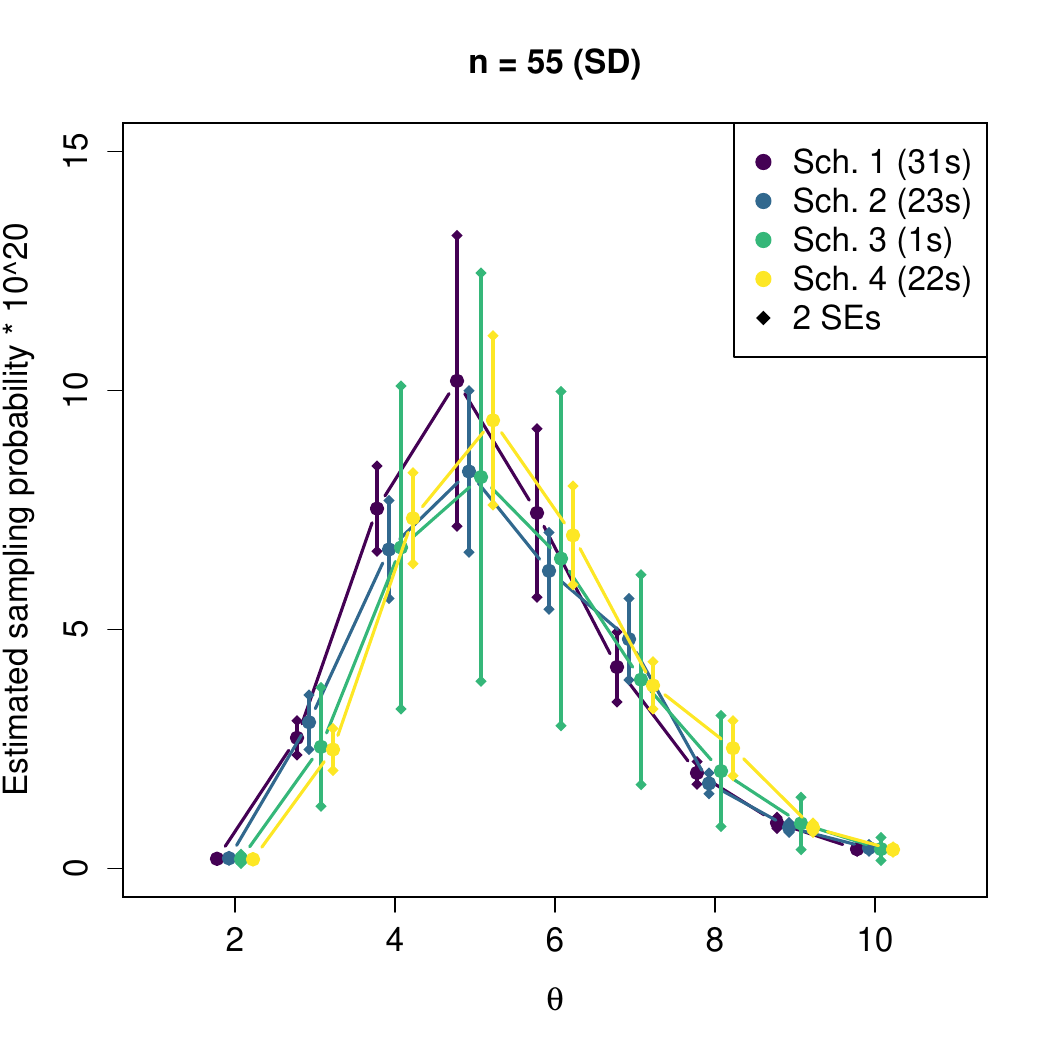}
\caption{}
\label{fig:likelihood-ism-55:c}
\end{subfigure}
\begin{subfigure}[t]{0.49\textwidth}
\includegraphics[width=\textwidth]{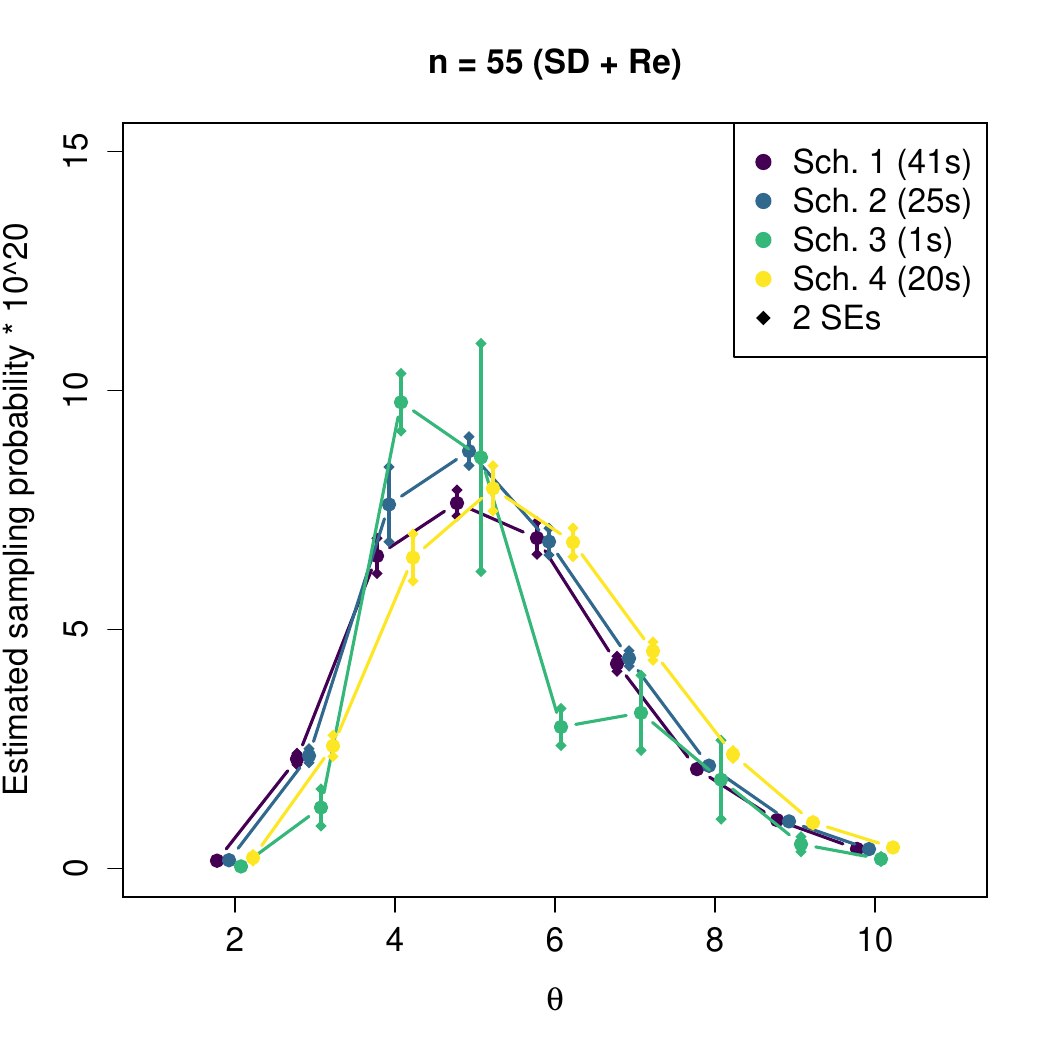}
\caption{}
\label{fig:likelihood-ism-55:d}
\end{subfigure}
\caption{Sampling probability estimates for the $n = 55$ data set from the HUW and SD proposals, simulated using the four schedules of replicates with $\gamma = 10^3$ and $\Gamma = 10^5$, independently for $\theta \in \{2, 3, \ldots, 10\}$.
Replicates in panels \ref{fig:likelihood-ism-55:b} and \ref{fig:likelihood-ism-55:d} were resampled in the way described in the caption of Figure \ref{fig:resampling}.
Standard errors for schedule 2, and for every schedule with resampling, were computed using the unbiased method of \cite{chan:2013}.
Each y-axis is expressed in units of $10^{-20}$ to aid visualisation, and small horizontal offsets have been artificially added to all four curves in each panel for visual clarity.}
\label{fig:likelihood-ism-55}
\end{figure}

Figure \ref{fig:likelihood-ism-550} shows that similar conclusions hold for a larger data set with $n = 550$ samples. 
It also illustrates the difference in computational cost between the HUW and SD proposals, which was already evident in the per-replicate analyses in \eqref{eq:huw_cost} and \eqref{eq:sd_cost}.
The gains in accuracy with the HUW proposal do not seem to compensate for its higher cost.
Comparing Figures \ref{fig:likelihood-ism-55} and \ref{fig:likelihood-ism-550} clearly shows that the performance of both the SD and HUW proposals deteriorates with increasing sample size, unlike in Figure \ref{fig:sd-surfaces} for the finite alleles model.
None of the schedules appear fully reliable, with plenty of evidence of Monte Carlo error and ill-calibrated standard error estimates.
Resampling appears to be beneficial, and indeed the means of schedules 1 and 2 nearly coincide when resampling takes place, though their standard error estimates still yield non-overlapping confidence intervals in some cases.

\begin{figure}[H]
\centering
\begin{subfigure}[t]{0.49\textwidth}
\includegraphics[width=\textwidth]{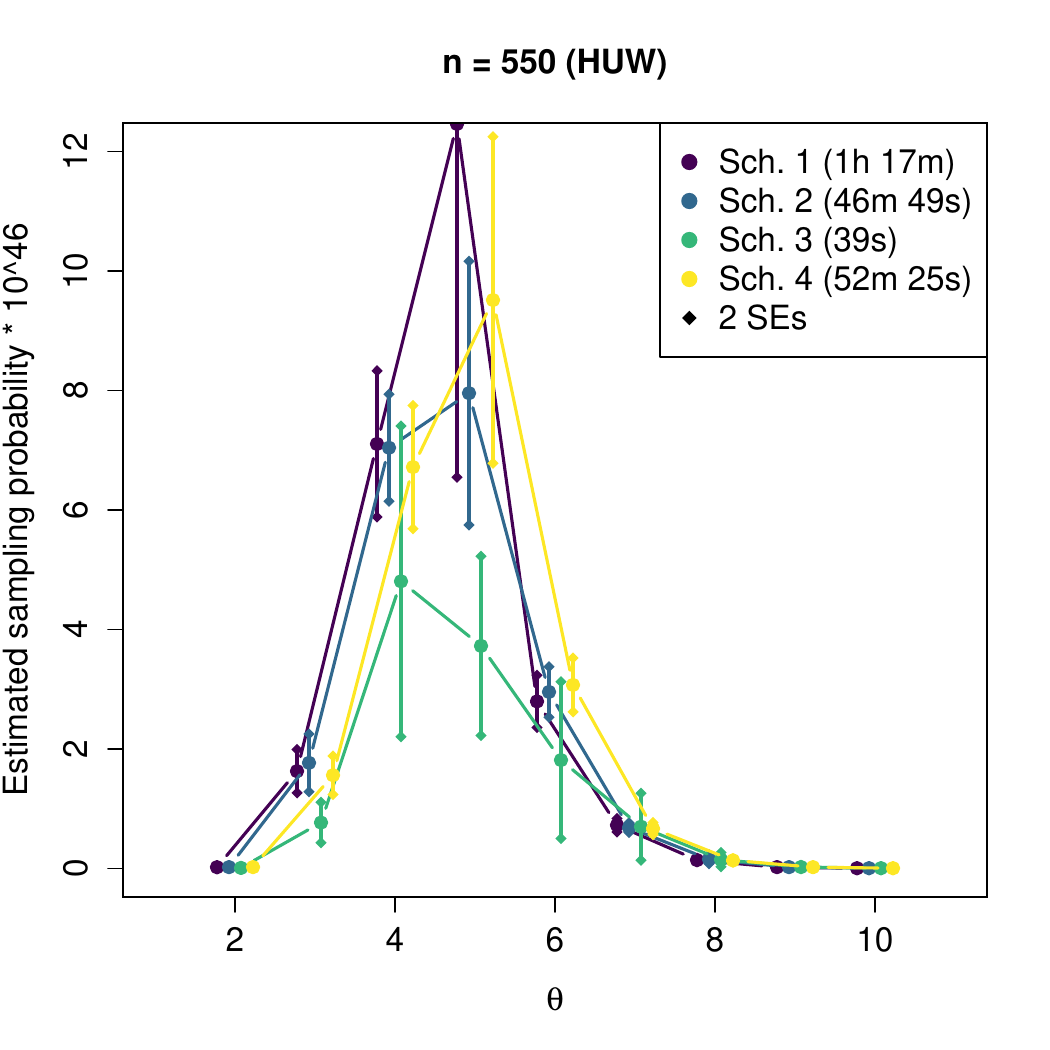}
\caption{}
\label{fig:likelihood-ism-550:a}
\end{subfigure}
\begin{subfigure}[t]{0.49\textwidth}
\includegraphics[width=\textwidth]{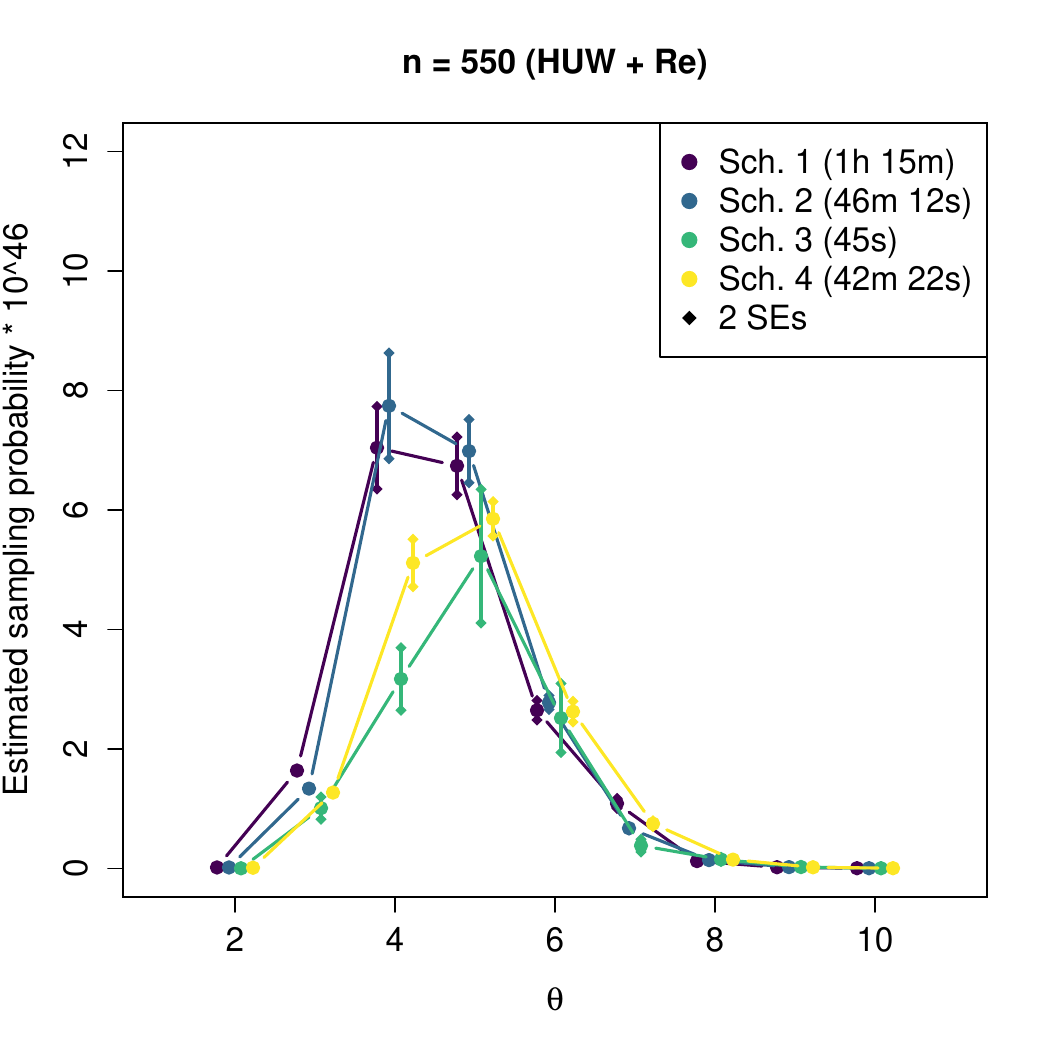}
\caption{}
\label{fig:likelihood-ism-550:b}
\end{subfigure}
\begin{subfigure}[t]{0.49\textwidth}
\includegraphics[width=\textwidth]{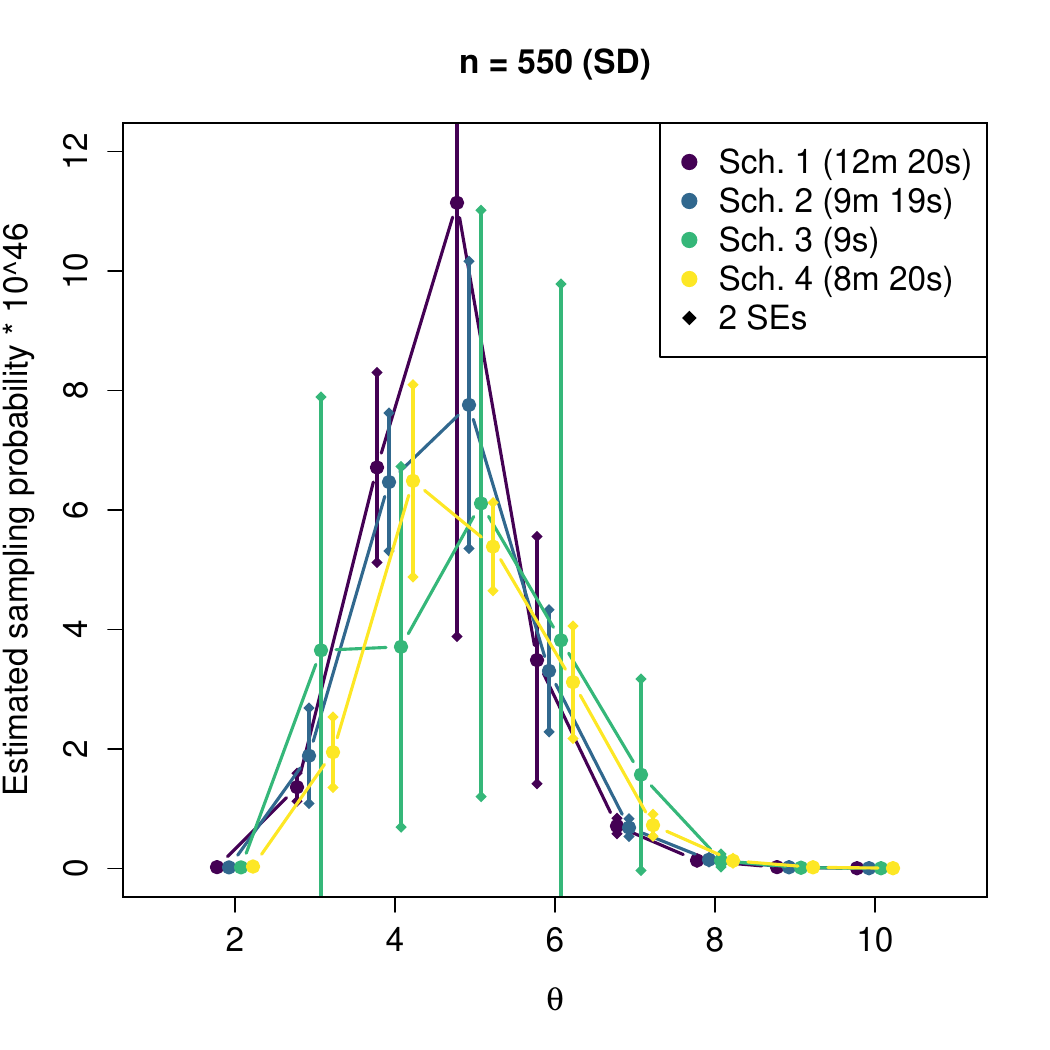}
\caption{}
\label{fig:likelihood-ism-550:c}
\end{subfigure}
\begin{subfigure}[t]{0.49\textwidth}
\includegraphics[width=\textwidth]{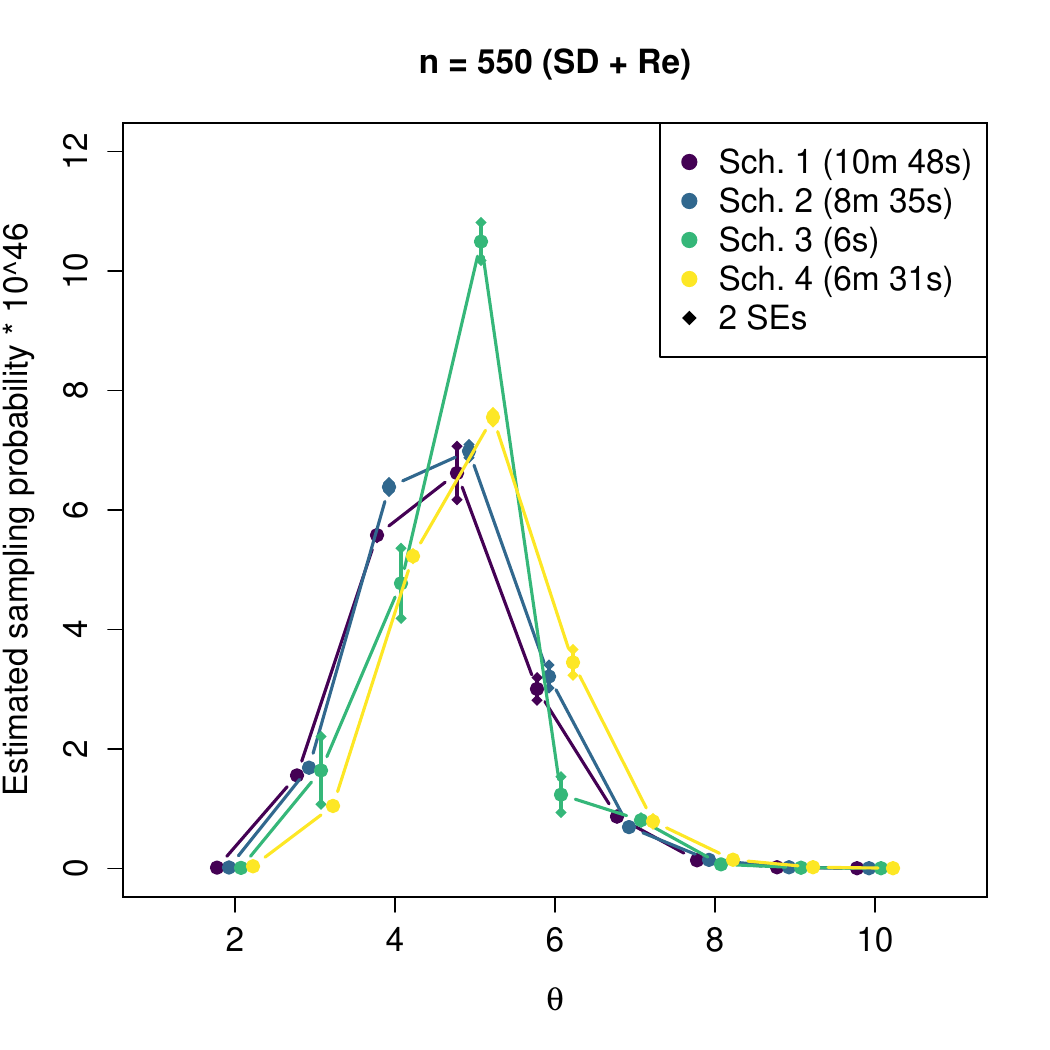}
\caption{}
\label{fig:likelihood-ism-550:d}
\end{subfigure}
\caption{Sampling probability estimates for the $n = 550$ data set from the HUW and SD proposals, simulated using the four schedules of replicates with $\gamma = 2 \times 10^3$ and $\Gamma = 2 \times 10^5$, independently for $\theta \in \{2, 3, \ldots, 10\}$.
Replicates in panels \ref{fig:likelihood-ism-550:b} and \ref{fig:likelihood-ism-550:d} were resampled in the way described in the caption of Figure \ref{fig:resampling}.
Standard errors for schedule 2, and for every schedule with resampling, were computed using the unbiased method of \cite{chan:2013}.
Eaxh y-axis is expressed in units of $10^{-46}$ to aid visualisation, and small horizontal offsets have been artificially added to all four curves in each panel for visual clarity.}
\label{fig:likelihood-ism-550}
\end{figure}

%%%%%%%%%%%%%%%%%%%%%%%%%%%%%%%%%%%%%%%%%%%%%%%%%%%%%%%%%%%%%%%%%%%
\section{Proofs}
\label{sect:proofs}
%%%%%%%%%%%%%%%%%%%%%%%%%%%%%%%%%%%%%%%%%%%%%%%%%%%%%%%%%%%%%%%%%%%

\subsection{Convergence of the cost sequence - Proof of Theorem \ref{thm:weak_conv}}
\label{proof:weak_conv}

The proof of Theorem \ref{thm:weak_conv} follows the steps of the proof of \cite[Theorem 2.1]{favero2024}, the difference being the additional cost component which leads to more complicated expressions and requires an extension of the technical framework and additional assumptions. 

\subsubsection{Technical framework and additional notation}
The scaled mutation probabilities in \eqref{eq:lim_trans_prob}, and consequently the intensities $\lambda_{ij}$ of the limiting Poisson processes of Theorem \ref{thm:weak_conv}, 
explode near the boundary 
    $
    \Omega_0 :=\{\y=(y_1,\dots,y_d) : y_i=0 \text{ for some } i \}.
    $
To address this problem, we define an appropriate state space for the limiting process and a specific metric under which compact sets are bounded away from the boundary $\Omega_0$. This is a straightforward generalisation of the technical framework of \cite{favero2024}. 

For the limiting process $\Z$, we thus consider the state space $E= \mathbb{R}_{+} \times E_1 \times \mathbb{N}^{d^2} $, where $E_1=(0,\infty]^d$. We equip $E$ with the product metric 
$\metric=  \norma{\cdot}\oplus\metric_1\oplus  \norma{\cdot}$, where $\psi_1(\y_1,\y_2)=\norma{1/\y_1-1/\y_2}$, 
with component-wise inversion and with the inverse of $\infty$ being $0$.
%For completeness, we naturally define $\Z$ to be constantly equal to its initial state, if one of the components of $\Y(0)$ is equal to $\infty$, although this type of initial condition does not appear in the relevant setting of Theorem \ref{thm:weak_conv}.
Note that, in  $E_1$, the roles of $0$ and $\infty$ are reversed component-wise, the metric $\metric_1$ is equivalent to the Euclidean metric away from the boundary $\Omega_0$ and from infinity,  and compact sets are bounded away from $\Omega_0$. 

Let $C_c^{\infty}(E)$ and $\hat{C}(E)$ be the spaces of real-valued continuous functions on $(E,\metric)$ that are, respectively, smooth with compact support or vanishing at infinity.
In $(E,\metric)$, functions with compact support are equal to zero near $\Omega_0$ in the $E_1$-component and near the classical infinity, in the other components. Similarly, functions vanishing at infinity, vanish towards $\Omega_0$ in the $E_1$-component and towards infinity, in the classical sense, in the other components.
For further explanations and properties of this state space and related functions we refer to \cite[Appendix A.2]{favero2024}.

Furthermore, let 
$E\nth= \mathbb{R}_+ \times \frac{1}{n}\mathbb{N}^d\setminus\{\boldsymbol{0}\}\times \mathbb{N}^{d^2} $ 
be the state space of $\Z\nth$,
and let $\eta_n $ map any function on $E$ into its restriction on $E\nth$, with value zero on $\mathbb{R}_+ \times \Omega_0 \times \mathbb{N}^{d^2}$.

\subsubsection{Convergence of generators (PIM)}
We now rigorously state and prove the convergence of generators which was  explained heuristically in Section \ref{sect:heuristic_conv}.
We assume parent-independent mutations here so that the backward transition probabilities are explicitly known, and we deal with the general mutation case in the last part of the proof. 

Let $A\nth $ be the infinitesimal generator of $\tilde{\Z}\nth$,  defined in \eqref{eq:An}, and let and $A $ be the infinitesimal generator of $\Z$, defined in \eqref{eq:A}.   
That the infinitesimal generator of $\Z$ is indeed $A$ is heuristically explained in Section \ref{sect:heuristic_conv}, the rigorous proof, which we omit, is analogous to the one in \cite[Appendix A.3]{favero2024}.

To prove convergence of generators, we need to prove that, for any given  $f\in C_c^{\infty}(E)$,
    \begin{equation}
    \label{eq:An2A}
    \lim_{n\to\infty}
    \sup_{(c,\y,\m)\in E\nth}
    \left|A\nth \eta_n f(c,\y,\m) - \eta_n A f (c,\y,\m)\right|
    =0.
    \end{equation} 
Since $f$ has compact support in $(E,\metric)$, there exist $\delta,M>0$ such that the support of $f$ is contained in the compact set 
    \begin{align*}
    K=\{(c,\y,\m)\in E: 
    y_i\geq \delta, 
    c\leq M, m_{ij}\leq M,\forall i,j \in \{1,\dots,d\}\}.
    \end{align*}
Let $K_1$ be the projection   of $K$ on $E_1$.
Assumption \ref{assumption:costs} implies 
    \begin{align}
    \label{eq:assumption}
     \lim_{n\to\infty}
     \sup_{\y\in E_1\nth\cap K_1}
     \left|n(c\nth(\e_j\mid\y)-1)-a_j(\y)\right| =0
     , \qquad
     \lim_{n\to\infty}
     \sup_{\y\in E_1\nth\cap K_1}
    \left|c\nth(\e_j-\e_i\mid\y)-b_{ij}(\y)\right|=0
     ,
    \end{align}
for $i,j \in \{1,\dots,d\}$.
Furthermore,
in  \cite[Proof of Theorem 2.1]{favero2024} it is shown, in the PIM case, that
    \begin{align}
    \label{eq:rho_uniform}
     \lim_{n\to\infty}
     \sup_{\y\in E_1\nth\cap K_1}
     \left|\rho\nth(\e_j|\y)-\frac{y_j}{\norm{\y}}\right| =0
    , \qquad
     \lim_{n\to\infty}
     \sup_{\y\in E_1\nth\cap K_1}
    \left|n\rho\nth(\e_j-\e_i|\y)-\lambda_{ij}(\y)\right|=0,   
    \end{align}
for $i,j \in \{1,\dots,d\}$.

To prove \eqref{eq:An2A}, we first take $(c,\y,\m)\in E\nth\cap K^\complement$. Then, $f=Af=0$ in a neighbourhood of $(c,\y,\m)$. 
If also 
$\left(c\, c\nth(\mathbf{e}_j|\y), \y\nth -\frac{1}{n}\mathbf{e}_j,\m\right)$
and
$\left(c\, c\nth (\mathbf{e}_j-\mathbf{e}_i|\y), \y -
    \frac{1}{n}\mathbf{e}_j+\frac{1}{n}\mathbf{e}_i,\m+\mathbf{e}_{ij}\right)$
belong to 
$E\nth\cap K^\complement$,
for all $i,j \in \{1,\dots,d\}$, $n\in\mathbb{N}$,
then 
$A\nth\eta_nf(c,\y,\m)=\eta_n A f(c,\y,\m)=0.$
Otherwise, it must be that $m_{ij}<M, i,j \in \{1,\dots,d\}$, and one of the following two cases occurs:
\begin{enumerate}
\item For a unique $i_0$ and some $n$, $\delta-1/n\leq y_{i_0}<\delta$, while $y_j\geq \delta$ for all $j\neq i_0$ and $c\leq M$, $c c\nth(\e_j\mid\y)\leq M$, $c c\nth(\e_j-\e_i\mid\y)\leq M$, $i,j \in \{1,\dots,d\}$;
\item $y_j\geq \delta$ for all $j \in \{1,\dots,d\}$, $c>M$, and,
for some $j$ and/or $i$, 
$c c\nth(\e_j\mid\y)\leq M$, and/or $c c\nth(\e_j-\e_i\mid\y)\leq M$.
\end{enumerate}
In both cases, $A\nth\eta_nf(c,\y,\m)$ is different from zero, but converges uniformly to $0$ because $\y\in K_1$, $b_{ij}\geq 1, \; i,j \in \{1,\dots,d\}$, and because of \eqref{eq:assumption}, \eqref{eq:rho_uniform}, and the properties of $f$. 

Now, we take $(c,\y,\m)\in E\nth\cap K$ and find a bound for $\left|A\nth \eta_n f(c,\y,\m) - \eta_n A f (c,\y,\m)\right|$.
First, note that, for $j \in \{1,\dots,d\},$ there exist $\bar{c}_j$, with $|\bar{c}_j-c|\leq |c c\nth(\e_j\mid\y)-c|$, and $s_j$, with $|s_j|\leq 1/n$, such that
    \begin{align*}
    &f\left(c c\nth(\e_j\mid\y),\y-\frac{1}{n}\e_j,\m\right)-f(c,\y,\m)
    \\&\qquad = 
    \partial_cf(\bar{c}_j,\y-s_j\e_j,\m) c( c\nth(\e_j\mid\y)-1) 
    - \frac{1}{n} \partial_{y_j}f (\bar{c}_j,\y-s_j\e_j,\m) .
    \end{align*}
Therefore,
	\begin{align}
	\nonumber
	&
	\left| 
	A\nth \eta_n f(c,\y,\m) -\eta_n Af(c,\y,\m)
	\right| 
	\\
	\label{An-A.1c}
	&\leq
	c\sum_{j=1}^d
	\left| 
	\partial_cf(\bar{c}_j,\y-s_j\e_j,\m) ( c\nth(\e_j\mid\y)-1) 
	\rho\nth(\mathbf{e}_j|\y)- 
	\partial_cf(c,\y,\m) a_j(\y) \frac{y_j}{\norm{\y}}
	\right| 
	\\ \label{An-A.1y}
	&\qquad+
	\sum_{j=1}^d
	\left| 
	\partial_{ y_j}	f(\bar{c}_j,\y-s_j\e_j,\m)
	\rho\nth(\mathbf{e}_j|\y) 
	 -\partial_{ y_j}f(c,\y,\m) \frac{y_j}{\norm{\y}}
	\right| 
	\\ \nonumber
	&\qquad+
	\sum_{i,j=1}^d
	\Bigg|
	f\left(c\, c\nth (\e_j-\e_i\mid\y), 
	\y -\frac{1}{n}\mathbf{e}_j+\frac{1}{n}\mathbf{e}_i,\m+\mathbf{e}_{ij}\right)
	n \rho\nth(\mathbf{e}_j-\mathbf{e}_i|\y)
	- 
	 \\  \label{An-A.2a}
	& \qquad \quad \quad \quad - 
	f(cb_{ij}(\y),\y,\m+\mathbf{e}_{ij}) \lambda_{ij}(\y)
	\Bigg| 
	\\ \label{An-A.2b}
	&\qquad
	+
	|f(c,\y,\m)| \sum_{i,j=1}^d 
	| n \rho\nth(\mathbf{e}_j-\mathbf{e}_i|\y)- \lambda_{ij}(\y)|.
	\end{align}

The $j^{th}$ term of the sum  \eqref{An-A.1c} is bounded by, using the mean value theorem, 
    \begin{align*}
    &M \left[M\norminfty{\partial_cf} |c\nth(\e_j\mid\y)-1| + \frac{1}{n}\norminfty{\partial_{y_j}\partial_c f}\right]
    n\left|  c\nth(\e_j\mid\y)-1 \right| 
    \\
    &\qquad+
    M \norminfty{\partial_c f} 
    \left|n( c\nth(\e_j\mid\y)-1) \rho\nth(\mathbf{e}_j|\y)- 
     a_j(\y)\frac{y_j}{\norm{\y}} \right|,
    \end{align*}
the supremum of which, over $y\in E_1\nth\cap K_1$, vanishes as $n\to \infty$, by   \eqref{eq:assumption}, \eqref{eq:rho_uniform}, and since $a$ is bounded on compact sets. 

The $j^{th}$ term of the sum 
\eqref{An-A.1y} is bounded by 
    \begin{align*}
    \left| 
	\partial_{ y_j}	f(\bar{c}_j,\y-s_j\e_j,\m)
	-\partial_{ y_j}f(c,\y,\m)
	\right|
	+
	\norminfty{\partial_{ y_j}f}
	\left|
	\rho\nth(\mathbf{e}_j|\y) - \frac{y_j}{\norm{\y}}
	\right|,
    \end{align*}
the supremum of which, over $y\in E_1\nth\cap K_1$, vanishes as $n\to \infty$,
since $\partial_{y_j}f$ is uniformly continuous, $|\bar{c}_j-c|\leq |c c\nth(\e_j\mid\y)-c|$, $|s_j|\leq \frac{1}{n}$ and by \eqref{eq:assumption}, \eqref{eq:rho_uniform}.

The $ij^{th}$ term in \eqref{An-A.2a} is bounded by
    \begin{align*}
    &\left|f\left(c\, c\nth (\e_j-\e_i\mid\y), 
	\y -\frac{1}{n}\mathbf{e}_j+\frac{1}{n}\mathbf{e}_i,\m+\mathbf{e}_{ij}\right)
	-
	f(cb_{ij}(\y),\y,\m+\mathbf{e}_{ij})
	\right|
	n \rho\nth(\mathbf{e}_j-\mathbf{e}_i|\y)
	\\&\qquad
	+
	\norminfty{ f} \left|n \rho\nth(\mathbf{e}_j-\mathbf{e}_i|\y)-\lambda_{ij}(\y) \right|,
    \end{align*}
the supremum of which, over $y\in E_1\nth\cap K_1$, vanishes as $n\to \infty$,
since $f$ is uniformly continuous and by \eqref{eq:assumption}, \eqref{eq:rho_uniform}.

Finally, the supremum of \eqref{An-A.2b} vanishes, as $n\to\infty$, by \eqref{eq:rho_uniform}, which concludes the proof of convergence of generators.

\subsubsection{Weak convergence (general mutation)}
The rest of the proof of Theorem \ref{thm:weak_conv} now follows from the same arguments as in \cite{favero2024}. We report a brief sketch here. 

Let $T\nth$ and $T$ be the semigroups associated to $\tilde{\Z}\nth$ and $\Z$ respectively. The convergence of generators, which holds in the PIM case, implies the following convergence of semigroups:  
for all $f\in\hat{C}(E)$, for all $t\geq 0$,
    \begin{equation}
    \lim_{n\to\infty}
    \sup_{(c,\y,\m)\in E\nth}
    \left|(T\nth) ^{\lfloor{tn}\rfloor} 
    \eta_n f(c,\y,\m) - \eta_n T(t) f (c,\y,\m)\right|
    =0,
    \end{equation} 
see \cite[Sect. 5.2]{favero2024} for details.
The semigroup $T$ is not conservative,  in fact, the process $\Z$ exits the state space in a finite time (when $\Y$ reaches the origin). Using the classical technique of \citet[Ch.4]{Ethier1986}, $T$ is extended to a conservative (Feller) semigroup, while the state space is extended to include the so-called \textit{cemetery point}. The weak convergence of the processes then easily follows, proving Theorem \ref{thm:weak_conv} in the PIM case. See \cite[Sect. 4 and 5.3]{favero2024} for details. 

To prove the result in the general mutation case, we can use the change-of-measure argument developed in \cite[Sect. 3]{favero2024}. This consists of changing the measures so that, under the new measures, the originally parent-dependent mutations become parent-independent. Crucially, the Radon--Nikodym derivatives (likelihood ratios) of the changes of measure depend on the block-counting and mutation-counting components, $\Y\nth, \Y, \M\nth, \M$, not on the cost-counting components, $C\nth, C$, and thus are exactly the same as in \citep{favero2024}, where the cost-components are not considered. Then, the PIM results can be applied to complete the proof in the general case, see \cite[Sect. 5.4]{favero2024} for details. 

\qed

\subsection{Asymptotic cost of one GT step --  Proof of Proposition \ref{prop:GT_cost_expansion}}
\label{proof:GT_cost_expansion}
    \begin{align*}
    c_{\scriptscriptstyle GT}\nth(\vv\mid\y)
    &= 
    \sum_{\vv'}p(n \y\mid n\y -\vv') 
    =
    \sum_{i=1}^d p(n \y\mid n\y -\e_i)+ \sum_{i,j=1}^d p(n \y\mid n\y -\e_i+\e_j)
    \\&=
    \sum_{i=1}^d 
    \frac{n y_i -1}{n\norm{\y}-1 + \theta} 
    %\frac{n \norm{\y}-2}{n\norm{\y}-1} %same result anyways
    + \sum_{i,j=1}^d
    \frac{ny_j-1+\delta_{ij}}{n \norm{\y}} \frac{\theta P_{ji}}{n\norm{\y}-1 + \theta}
    \\&=
    \frac{n \norm{\y}-d}{n \norm{\y}-1+\theta} 
    + \frac{n \norm{\y}-d +\sum_{i=1}^d P_{ii}}{n \norm{\y}} \frac{\theta}{n \norm{\y}-1+\theta} 
    \\&=
    \frac{1}{n\norm{\y}-1+\theta} 
    \left[n \norm{\y}-d + \theta \frac{n \norm{\y}-d +\sum_{i=1}^d P_{ii}}{n \norm{\y}} \right]
    \\&=
    \left[\frac{1}{n} \frac{1}{\norm{\y}}+\frac{1}{n^2} \frac{1-\theta}{\norm{\y}^2}+o\left(\frac{1}{n^2}\right)\right]
    \left[n \norm{\y}-d + \theta + o(1) \right]
    \end{align*}
from which the result follows.

\qed

\subsection{Asymptotic cost of one SD step --  Proof of Proposition \ref{prop:SD_cost_expansion}}
\label{proof:SD_cost_expansion}
Using that 
    \begin{align*}
    \frac{1}{n\norm{\y}-1+\theta}=
    \frac{1}{n} \frac{1}{\norm{\y}}+\frac{1}{n^2} \frac{1-\theta}{\norm{\y}^2}+o\left(\frac{1}{n^2}\right)  ,
    \end{align*}
we obtain    
    \begin{align*}
    \hat{\pi}[i\mid n\y-\e_j ]
    &=
    \sum_{i'=1}^d \frac{ny_{i'}-\delta_{i'j}}{n\norm{\y}-1+\theta}
    \sum_{m=0}^\infty \left(\frac{\theta}{n\norm{\y}-1+\theta}\right)^m(P^m)_{i'i}
    \\&=
    \sum_{i'=1}^d \frac{ny_{i'}-\delta_{i'j}}{n\norm{\y}-1+\theta}
    \left[\delta_{i'i}
    +\frac{\theta}{n\norm{\y}-1+\theta} P_{i'i} + o\left(\frac{1}{n}\right)
    \right]
    \\&=
    \frac{ny_i-\delta_{ij}}{n\norm{\y}-1+\theta}
    +
    \frac{\theta}{(n\norm{\y}-1+\theta)^2} \sum_{i'=1}^d (ny_{i'}-\delta_{i'j}) P_{i'i}
    +o\left(\frac{1}{n}\right)
    \\&=
    \frac{y_i}{\norm{\y}}  
    - \frac{1}{n} \frac{\delta_{ij}}{\norm{\y}}
    +
    \frac{1}{n} \frac{y_i(1-\theta)}{\norm{\y}^2}
    +
    \frac{1}{n}  \frac{\theta}{\norm{\y}^2}
    \sum_{i'=1}^d y_{i'} P_{i'i}
    +o\left(\frac{1}{n}\right)
    \end{align*}
from which the result follows. 

\qed

\subsection{Weak convergence (under proposal distributions) -- Proof of Proposition \ref{prop:weak_conv_proposals}} 
\label{proof:weak_conv_proposals}

The infinitesimal generator of $\tilde{\Z}\nth$ under the GT or SD proposals can be obtained from the expression \eqref{eq:An} of the infinitesimal generator of $\tilde{\Z}\nth$ under the true distribution, by replacing $\rho\nth(\vv\mid\y) = p(n\y -\vv \mid n \y)$ with
$\rho_{GT}\nth(\vv\mid\y) = q_{GT}(n\y -\vv \mid n \y)$
or
$\rho_{SD}\nth(\vv\mid\y) = q_{SD}(n\y -\vv \mid n \y)$.

Using Proposition \ref{prop:GT_cost_expansion}, Definition \ref{def:forward_transitions}, and \eqref{eq:GT_proposal} for GT; and 
Proposition \ref{prop:SD_cost_expansion} and \eqref{eq:SD_proposal} for SD; 
it is straightforward to show that the first order approximation of the GT and SD transition probabilities corresponds to the first order approximation of the true transition probabilities. 
%(display \eqref{eq:lim_trans_prob}, \citet[Coroll. 5.2]{favero2022}). 
That is, assuming  $\y\nth\to\y\in\mathbb{R}_{+}^d$, we have 
    \begin{align*}
    &
   \lim_{n\to\infty} \rho_{GT}\nth (\e_j\mid \y\nth) 
    =
    \lim_{n\to\infty} \rho_{SD}\nth (\e_j\mid \y\nth)
    =
    \lim_{n\to\infty} \rho\nth (\e_j\mid \y\nth)
    =
    \frac{y_j}{\norm{\y}};
    \\
    & 
    \lim_{n\to\infty} n \rho_{GT}\nth (\e_j-\e_i\mid \y\nth)
    =
    \lim_{n\to\infty} n \rho_{SD}\nth (\e_j-\e_i\mid \y\nth)
    =
    \lim_{n\to\infty} n \rho\nth (\e_j-\e_i\mid \y\nth)
    =
    \lambda_{ij}(\y).
    \end{align*}
The convergence above is uniform in the sense of \eqref{eq:rho_uniform}. 
Then, the convergence of generators holds under the proposal distributions. The rest of the proof of Proposition \ref{prop:weak_conv_proposals} is then identical to that of Theorem \ref{thm:weak_conv}, without even the need for a change-of-measure argument, since the proposal transition probabilities are always explicit (as the transition probabilities in the PIM case). 

\qed

\subsection{Convergence of importance sampling weights -- Proof of Theorem \ref{thm:weights_convergence}}
\label{proof:weights_convergence}

By \cite[Theorem 4.3]{favero2022}, when $\y_0\nth\to\y_0$, as $n\to\infty$,  we have that
    $$
    n^{d-1}p(n\y_0\nth)
    \to \norm{\y_0}^{1-d}
    \tilde{p}\left(\frac{\y_0}{\norm{\y_0}}\right)
    = \tilde{p}\left(\y_0\right),
    $$
where $\tilde{p}$ is the (smooth) stationary density of the dual Wright-Fisher diffusion. By Theorem \ref{thm:weak_conv}, or by \cite[Theorem 2.1]{favero2024}, we know 
    $
    \Y\nth(\lfloor{tn}\rfloor{})
    \xrightarrow[]{\mathcal{D}}  
    \Y(t)=\y_0(1-t) 
    $, 
thus, by applying again \cite[Theorem 4.3]{favero2022} we obtain 
    $$
    n^{d-1}p(n\Y\nth(\lfloor{tn}\rfloor{}))
    \xrightarrow[n\to\infty]{\mathcal{D}} 
    \norm{\Y(t)}^{1-d}
    \tilde{p}\left(\frac{\Y(t)}{\norm{\Y(t)}}\right)
    =
   (1-t)^{1-d}  \tilde{p}\left(\y_0\right).
    $$
The first convergence is proven. 

\qed

\subsubsection{Griffiths--Tavar\'e}
By Theorem \ref{thm:weak_conv} and Proposition \ref{prop:GT_cost_expansion},
    \begin{align*}
    C\nth_{GT}(\lfloor{tn}\rfloor{})
    \xrightarrow[n\to\infty]{\mathcal{D}}
    C_{GT}(t)
    = &
    \exp\left\{\sum_{i=1}^d y_{0,i} \int_0^t  \frac{1-d}{\norm{\y_0(1-u)}}du \right\}
    \\ =&
    \exp\left\{\int_0^t \frac{1-d}{1-u}du \right\}
    \\ =&
    \exp\left\{(d-1)\log(1-t) \right\}
    \\ =& (1-t)^{d-1} ,
    \end{align*}
which proves the convergence of costs. Then, by equation \ref{eq:normalised_weights_costs}, the convergence of the corresponding weights is also proven. 

\qed
 
\subsubsection{Stephens--Donnelly}
By Theorem \ref{thm:weak_conv} and Proposition \ref{prop:SD_cost_expansion},
    \begin{align*} 
    C\nth_{SD}(\lfloor{tn}\rfloor{}) 
    \xrightarrow[n\to\infty]{\mathcal{D}}
    C_{SD}(t)
    = \exp\left\{\sum_{i=1}^d y_{0,i} \int_0^t  \hat{a}_i\left(\y_0(1-u)\right)du \right\}
    =\exp\left\{\int_0^t \frac{1-d}{1-u}du \right\}
    =  (1-t)^{d-1},
    \end{align*}
since
    \begin{align*}
    \hat{a}_i(\y_0(1-u))
    = \frac{1-\theta}{1-u}
    -\frac{1}{y_{0,i}(1-u)}\left(1-\sum_{i'=1}^d y_{0,i'}\theta P_{i'i}\right),
    \end{align*}
and  
    \begin{align*}
    \sum_{i=1}^d y_{0,i}
    \hat{a}_i(\y_0(1-u))
    = 
    \frac{1-\theta}{1-u}
    -\frac{1}{1-u} \sum_{i=1}^d  
    \left(1-\sum_{i'=1}^d y_{0,i'}\theta P_{i'i}\right)
    =
    \frac{1}{1-u} \left[
    1-\theta -d + \theta
    \right].
    \end{align*}
This proves the convergence of costs. Then, by equation \ref{eq:normalised_weights_costs}, the convergence of the corresponding weights is also proven. 

\qed

\section{Discussion}\label{sect:discussion}

We have shown that the existing large-sample asymptotics for the coalescent developed by \cite{favero2024} can be extended to incorporate cost functionals of the coalescent.
%trees weighted by cost functions.
Particular choices of costs render the theory applicable to analysis of sequential importance sampling algorithms for the coalescent.
Importance sampling for the coalescent is notoriously difficult for large samples, and to our knowledge, our results are the first rigorous description of its behaviour.
They also create a connection between coalescent importance sampling and stochastic control approaches to rare event simulation, where the asymptotic analysis of a sequence of costs is a standard method.

We envisage several interesting directions to which our work can be extended.
Our exposition has focused on the coalescent as a model in population genetics, but it also finds applications as a prior in Bayesian nonparametrics and clustering \citep{gorur2008}.
Other models of coalescing and mutating lineages are also widespread in those settings, with the two-parameter Pitman--Yor process being a prominent example \citep{perman1992, pitman1997}.
Analogues of our scaling limit might hold for the Pitman--Yor process, or other Bayesian clustering models, and inform their use for large sample sizes as well.

In genetics, the coalescent is a robust model for a wide range of settings and organisms, but relies on a small variance of family sizes relative to population size.
If family sizes are heavily skewed, evolution can be more accurately described by multiple merger coalescents, in which more than two lineages can coalesce simultaneously \citep{donnelly1999, pitman1999, sagitov1999}, and more than one simultaneous coalescence can take place \citep{mohle2001, schweinsberg2000}.
Importance sampling methods for these types of models are available but are even less scalable as those for the standard coalescent \citep{birkner2008, birkner2011, Koskela2015}.
A similar scaling limit for multiple merger coalescents would be of mathematical interest, and could inform importance sampling methods for them as well.
If such a scaling limit exists, we expect it would incorporate macroscopic jumps towards the origin driven by multiple mergers.

Finally, modern data sets rarely consist of a single locus.
Hence it would be of interest to obtain a similar description of weighted ancestral recombination graphs, which are the multi-locus analogue of the coalescent.
Evolution at two unlinked loci would correspond to two independent copies of our limiting process. A scaling limit for two linked loci should be informative of how linkage creates correlation between the two copies of the limit process.
Such a result would be of mathematical interest, and could also inform importance sampling methods \citep{fearnhead2001} and more heuristic methods \citep{li:2003} for genomic inference.

\section*{Acknowledgements}
We would like to thank Henrik Hult for suggesting the initial idea that originated this project and for contributing to its early development. 
MF acknowledges the support of the Knut and Alice Wallenberg Foundation (Program for Mathematics, grant 2020.072).

\bibliography{mybib}{}

\begin{thebibliography}{54}
\providecommand{\natexlab}[1]{#1}
\providecommand{\url}[1]{\texttt{#1}}
\expandafter\ifx\csname urlstyle\endcsname\relax
  \providecommand{\doi}[1]{doi: #1}\else
  \providecommand{\doi}{doi: \begingroup \urlstyle{rm}\Url}\fi

\bibitem[Beaumont(2010)]{beaumont:2010}
M.~A. Beaumont.
\newblock Approximate {Bayesian} computation in evolution and ecology.
\newblock \emph{Annual Review of Ecology, Evolution, and Systematics},
  44\penalty0 (2):\penalty0 397--406, 2010.

\bibitem[Birkner and Blath(2008)]{birkner2008}
M.~Birkner and J.~Blath.
\newblock Computing likelihoods for coalescents with multiple collisions in the
  infinitely many sites model.
\newblock \emph{Journal of Mathematical Biology}, 57\penalty0 (3):\penalty0
  435--465, 2008.

\bibitem[Birkner et~al.(2011)Birkner, Blath, and Steinrücken]{birkner2011}
M.~Birkner, J.~Blath, and M.~Steinrücken.
\newblock Importance sampling for {Lambda}-coalescents in the infinitely many
  sites model.
\newblock \emph{Theoretical Population Biology}, 79\penalty0 (4):\penalty0
  155--173, 2011.

\bibitem[Blanchet et~al.(2012)Blanchet, Glynn, and Leder]{blanchet2012}
J.~Blanchet, P.~Glynn, and K.~Leder.
\newblock On {Lyapunov} inequalities and subsolutions for efficient importance
  sampling.
\newblock \emph{ACM Transactions on Modeling and Computer Simulation},
  22\penalty0 (3), 2012.

\bibitem[Chan and Lai(2013)]{chan:2013}
H.~P. Chan and T.~L. Lai.
\newblock A general theory of particle filters in hidden {Markov} models and
  some applications.
\newblock \emph{Annals of Statistics}, 41:\penalty0 2877--2904, 2013.

\bibitem[Chatterjee and Diaconis(2018)]{chatterjee2018sample}
S.~Chatterjee and P.~Diaconis.
\newblock The sample size required in importance sampling.
\newblock \emph{The Annals of Applied Probability}, 28\penalty0 (2):\penalty0
  1099--1135, 2018.

\bibitem[Chen et~al.(2005)Chen, Xie, and Liu]{chen2005}
Y.~Chen, J.~Xie, and J.~S. Liu.
\newblock Stopping-time resampling for sequential {Monte Carlo} methods.
\newblock \emph{Journal of the Royal Statistical Society: Series B},
  67:\penalty0 199--217, 2005.

\bibitem[Chopin and Papaspiliopoulos(2020)]{chopin:2020}
N.~Chopin and O.~Papaspiliopoulos.
\newblock \emph{An Introduction to Sequential Monte Carlo}.
\newblock Springer Cham, 2020.

\bibitem[{De Iorio} and Griffiths(2004)]{deiorio2004}
M.~{De Iorio} and R.~C. Griffiths.
\newblock Importance sampling on coalescent histories. {I}.
\newblock \emph{Advances in Applied Probability}, 36\penalty0 (2):\penalty0
  417–433, 2004.

\bibitem[Donnelly and Kurtz(1999)]{donnelly1999}
P.~Donnelly and T.~G. Kurtz.
\newblock Particle representations for measure-valued population models.
\newblock \emph{Annals of Probability}, 27\penalty0 (1):\penalty0 166--205,
  1999.

\bibitem[Doucet and Johansen(2011)]{doucet2011}
A.~Doucet and A.~M. Johansen.
\newblock A tutorial on particle filtering and smoothing: fifteen years later.
\newblock In D.~Crisan and B.~Rozovsky, editors, \emph{The Oxford Handbook of
  Nonlinear Filtering}. Oxford University Press, 2011.

\bibitem[Dupuis and Wang(2004)]{dupuis2004}
P.~Dupuis and H.~Wang.
\newblock Importance sampling, large deviations, and differential games.
\newblock \emph{Stochastics and Stochastic Reports}, 76\penalty0 (6):\penalty0
  481--508, 2004.

\bibitem[Ethier and Kurtz(1986)]{Ethier1986}
S.~N. Ethier and T.~G. Kurtz.
\newblock \emph{Markov processes: characterization and convergence}, volume
  282.
\newblock John Wiley \& Sons, 1986.

\bibitem[Fan and Wakeley(2024)]{fan2024}
W.~T.~L. Fan and J.~Wakeley.
\newblock Latent mutations in the ancestries of alleles under selection.
\newblock \emph{Theoretical Population Biology}, 158:\penalty0 1--20, 2024.

\bibitem[Favero and Hult(2022)]{favero2022}
M.~Favero and H.~Hult.
\newblock {Asymptotic behaviour of sampling and transition probabilities in
  coalescent models under selection and parent dependent mutations}.
\newblock \emph{Electronic Communications in Probability}, 27:\penalty0 1--13,
  2022.

\bibitem[Favero and Hult(2024)]{favero2024}
M.~Favero and H.~Hult.
\newblock {Weak convergence of the scaled jump chain and number of mutations of
  the Kingman coalescent}.
\newblock \emph{Electronic Journal of Probability}, 29:\penalty0 1--22, 2024.

\bibitem[Favero and Jenkins(2025)]{favero2025}
M.~Favero and P.~A. Jenkins.
\newblock Sampling probabilities, diffusions, ancestral graphs, and duality
  under strong selection.
\newblock \emph{Electronic Journal of Probability}, 30:\penalty0 1--42, 2025.

\bibitem[Fearnhead(2008)]{fearnhead2008}
P.~Fearnhead.
\newblock Computational methods for complex stochastic systems: a review of
  some alternatives to {MCMC}.
\newblock \emph{Statistics and Computing}, 18:\penalty0 151--171, 2008.

\bibitem[Fearnhead and Donnelly(2001)]{fearnhead2001}
P.~Fearnhead and P.~Donnelly.
\newblock Estimating recombination rates from population genetic data.
\newblock \emph{Genetics}, 159\penalty0 (3):\penalty0 1299--1318, 2001.

\bibitem[Gorur and Teh(2008)]{gorur2008}
D.~Gorur and Y.~Teh.
\newblock An efficient sequential monte carlo algorithm for coalescent
  clustering.
\newblock In D.~Koller, D.~Schuurmans, Y.~Bengio, and L.~Bottou, editors,
  \emph{Advances in Neural Information Processing Systems}, volume~21. Curran
  Associates, Inc., 2008.

\bibitem[{Griffiths} and
  {Tavar\'e}(1994{\natexlab{a}})]{griffiths1994ancestral}
R.~C. {Griffiths} and S.~{Tavar\'e}.
\newblock Ancestral inference in population genetics.
\newblock \emph{Statistical Science}, 9\penalty0 (3):\penalty0 307--319,
  1994{\natexlab{a}}.

\bibitem[{Griffiths} and
  {Tavar\'e}(1994{\natexlab{b}})]{griffiths1994simulating}
R.~C. {Griffiths} and S.~{Tavar\'e}.
\newblock Simulating probability distributions in the coalescent.
\newblock \emph{Theoretical Population Biology}, 46\penalty0 (2):\penalty0
  131--159, 1994{\natexlab{b}}.

\bibitem[Griffiths et~al.(2008)Griffiths, Jenkins, and Song]{griffiths2008}
R.~C. Griffiths, P.~A. Jenkins, and Y.~S. Song.
\newblock Importance sampling and the two-locus model with subdivided
  population structure.
\newblock \emph{Advances in Applied Probability}, 40\penalty0 (2):\penalty0
  473--500, 2008.

\bibitem[Hobolth et~al.(2008)Hobolth, Uyenoyama, and Wiuf]{hobolth2008}
A.~Hobolth, M.~K. Uyenoyama, and C.~Wiuf.
\newblock Importance sampling for the infinite sites model.
\newblock \emph{Statistical Applications in Genetics and Molecular Biology},
  7\penalty0 (1):\penalty0 32, 2008.

\bibitem[Jasra et~al.(2011)Jasra, {De Iorio}, and Chadeau-Hyam]{jasra2011}
A.~Jasra, M.~{De Iorio}, and M.~Chadeau-Hyam.
\newblock The time machine: a simulation approach for stochastic trees.
\newblock \emph{Proceedings of the Royal Society A}, 467:\penalty0 2350--2368,
  2011.

\bibitem[Jenkins(2012)]{jenkins2012stopping}
P.~A. Jenkins.
\newblock Stopping-time resampling and population genetic inference under
  coalescent models.
\newblock \emph{Statistical Applications in Genetics and Molecular Biology},
  11\penalty0 (1):\penalty0 Article 9, 2012.

\bibitem[Jenkins and Song(2009)]{jenkins2009}
P.~A. Jenkins and Y.~S. Song.
\newblock {Closed-form two-locus sampling distributions: accuracy and
  universality}.
\newblock \emph{Genetics}, 183\penalty0 (3):\penalty0 1087--1103, 11 2009.

\bibitem[Jenkins and Song(2010)]{jenkins2010}
P.~A. Jenkins and Y.~S. Song.
\newblock {An asymptotic sampling formula for the coalescent with
  recombination}.
\newblock \emph{The Annals of Applied Probability}, 20\penalty0 (3):\penalty0
  1005--1028, 2010.

\bibitem[Jenkins and Song(2012)]{jenkins2012}
P.~A. Jenkins and Y.~S. Song.
\newblock {Padé approximants and exact two-locus sampling distributions}.
\newblock \emph{The Annals of Applied Probability}, 22\penalty0 (2):\penalty0
  576--607, 2012.

\bibitem[Jenkins et~al.(2015)Jenkins, Fearnhead, and Song]{jenkins2015}
P.~A. Jenkins, P.~Fearnhead, and Y.~Song.
\newblock {Tractable diffusion and coalescent processes for weakly correlated
  loci}.
\newblock \emph{Electronic Journal of Probability}, 20:\penalty0 1--25, 2015.

\bibitem[Kelleher et~al.(2019)Kelleher, Wong, Wohns, Fadil, Albers, and
  McVean]{kelleher2019}
J.~Kelleher, Y.~Wong, A.~W. Wohns, C.~Fadil, P.~K. Albers, and G.~McVean.
\newblock Inferring whole-genome histories in large population datasets.
\newblock \emph{Nature Genetics}, 51:\penalty0 1330--1338, 2019.

\bibitem[Kingman(1982)]{kingman1982b}
J.~Kingman.
\newblock The coalescent.
\newblock \emph{Stochastic Processes and their Applications}, 13\penalty0
  (3):\penalty0 235 -- 248, 1982.

\bibitem[Kong et~al.(1994)Kong, Liu, and Hong]{kong1994}
A.~Kong, J.~S. Liu, and W.~H. Hong.
\newblock Sequential imputations and {Bayesian} missing data problems.
\newblock \emph{Journal of the American Statistical Association}, 89:\penalty0
  278--288, 1994.

\bibitem[Koskela et~al.(2015)Koskela, Jenkins, and Spanò]{Koskela2015}
J.~Koskela, P.~Jenkins, and D.~Spanò.
\newblock Computational inference beyond {Kingman's} coalescent.
\newblock \emph{Journal of Applied Probability}, 52\penalty0 (2):\penalty0
  519--537, 06 2015.

\bibitem[Lawson et~al.(2012)Lawson, Hellenthal, Myers, and Falush]{lawson:2012}
D.~J. Lawson, G.~Hellenthal, S.~Myers, and D.~Falush.
\newblock Inference of population structure using dense haplotype data.
\newblock \emph{PLOS Genetics}, 8\penalty0 (1):\penalty0 e1002453, 2012.

\bibitem[Lee and Whiteley(2018)]{lee:2018}
A.~Lee and N.~Whiteley.
\newblock Variance estimation in the particle filter.
\newblock \emph{Biometrika}, 105\penalty0 (3):\penalty0 609--625, 2018.

\bibitem[Li and Stephens(2003)]{li:2003}
N.~Li and M.~Stephens.
\newblock Modeling linkage disequilibrium and identifying recombination
  hotspots using single-nucleotide polymorphism data.
\newblock \emph{Genetics}, 165\penalty0 (4):\penalty0 2213--2233, 2003.

\bibitem[Lundstrom et~al.(1992)Lundstrom, Tavar\'e, and Ward]{lundstrom:1992}
R.~Lundstrom, S.~Tavar\'e, and R.~H. Ward.
\newblock Estimating substitution rates from molecular data using the
  coalescent.
\newblock \emph{Proceedings of the National Academy of Sciences}, 89:\penalty0
  5961--5965, 1992.

\bibitem[Marjoram and Tavar\'e(2006)]{marjoram:2006}
P.~Marjoram and S.~Tavar\'e.
\newblock Modern computational appraoches for analysing molecular genetic
  variation data.
\newblock \emph{Nature Reviews Genetics}, 7:\penalty0 759--770, 2006.

\bibitem[Möhle and Sagitov(2001)]{mohle2001}
M.~Möhle and S.~Sagitov.
\newblock A classification of coalescent processes for haploid exchangeable
  population models.
\newblock \emph{Annals of Probability}, 29:\penalty0 1547--1562, 2001.

\bibitem[Perman et~al.(1992)Perman, Pitman, and Yor]{perman1992}
M.~Perman, J.~Pitman, and M.~Yor.
\newblock Size-biased sampling of {Poisson} point processes and excursions.
\newblock \emph{Probability Theory and Related Fields}, 92\penalty0
  (1):\penalty0 21--39, 1992.

\bibitem[Pitman(1999)]{pitman1999}
J.~Pitman.
\newblock Coalescent with multiple collisions.
\newblock \emph{Annals of Probability}, 27:\penalty0 1870--1902, 1999.

\bibitem[Pitman and Yor(1997)]{pitman1997}
J.~Pitman and M.~Yor.
\newblock The two-parameter {Poisson}--{Dirichlet} distribution derived from a
  stable subordinator.
\newblock \emph{Annals of Probability}, 25\penalty0 (2):\penalty0 855--900,
  1997.

\bibitem[Sagitov(1999)]{sagitov1999}
S.~Sagitov.
\newblock The general coalescent with asynchronous mergers of ancestral lines.
\newblock \emph{Journal of Applied Probability}, 36:\penalty0 1116--1125, 1999.

\bibitem[Sawyer et~al.(1987)Sawyer, Dykhuizen, and Hartl]{sawyer1987}
S.~A. Sawyer, D.~E. Dykhuizen, and D.~L. Hartl.
\newblock Confidence interval for the number of selectively neutralamino acid
  polymorphisms.
\newblock \emph{Proceedings of the National Academy of Sciences}, 84:\penalty0
  6225--6228, 1987.

\bibitem[Schweinsberg(2000)]{schweinsberg2000}
J.~Schweinsberg.
\newblock Coalescents with simultaneous multiple collisions.
\newblock \emph{Electronic Journal of Probability}, 5:\penalty0 Article 12,
  2000.

\bibitem[Song et~al.(2006)Song, Lyngs\o, and Hein]{song2006}
Y.~S. Song, R.~Lyngs\o, and J.~Hein.
\newblock Counting all possible ancestral configurations of sample sequences in
  population genetics.
\newblock \emph{IEEE/ACM Transactions on Computational Biology and
  Bioinformatics}, 3:\penalty0 239--251, 2006.

\bibitem[Stephens(2007)]{Stephens2007}
M.~Stephens.
\newblock Inference under the coalescent.
\newblock In D.~Balding, M.~Bishop, and C.~Cannings, editors, \emph{Handbook of
  Statistical Genetics}, chapter~26, pages 878--908. Wiley, Chichester, UK,
  2007.

\bibitem[Stephens and Donnelly(2000)]{stephens2000}
M.~Stephens and P.~Donnelly.
\newblock Inference in molecular population genetics.
\newblock \emph{Journal of the Royal Statistical Society: Series B (Statistical
  Methodology)}, 62\penalty0 (4):\penalty0 605--635, 2000.

\bibitem[Stephens and Donnelly(2003)]{stephens2003}
M.~Stephens and P.~Donnelly.
\newblock Ancestral inference in population genetics models with selection
  (with discussion).
\newblock \emph{Australian \& New Zealand Journal of Statistics}, 45\penalty0
  (4):\penalty0 395--430, 12 2003.

\bibitem[Wakeley(2008)]{wakeley2008}
J.~Wakeley.
\newblock {Conditional gene genealogies under strong purifying selection}.
\newblock \emph{Molecular Biology and Evolution}, 25\penalty0 (12):\penalty0
  2615--2626, 09 2008.

\bibitem[Wakeley and Sargsyan(2009)]{wakeley2009}
J.~Wakeley and O.~Sargsyan.
\newblock The conditional ancestral selection graph with strong balancing
  selection.
\newblock \emph{Theoretical Population Biology}, 75 4:\penalty0 355--64, 2009.

\bibitem[Ward et~al.(1991)Ward, Frazier, Dew, and P\"a\"abo]{ward1991}
R.~H. Ward, B.~L. Frazier, K.~Dew, and S.~P\"a\"abo.
\newblock Extensive mitochondrial diversity within a single {Amerindian} tribe.
\newblock \emph{Proceedings of the National Academy of Sciences}, 88:\penalty0
  8720--8724, 1991.

\bibitem[Watterson(1975)]{watterson1975}
G.~A. Watterson.
\newblock On the number of segregating sites in genetical models without
  recombination.
\newblock \emph{Theoretical Population Biology}, 7:\penalty0 256--276, 1975.

\end{thebibliography}
\bibliographystyle{abbrvnat}

%%%%%%%%%%%%%%%%%% Appendix %%%%%%%%%%%%%%%%%%%%%
%\clearpage
%\pagebreak
%%\nolinenumbers
%\begin{center}
%\bf{\LARGE Appendix}
%\vspace{10pt}
%\hrule
%\end{center}

%%%%%%%%%%% Prefix "A" to everything %%%%%%%%%%%%
%\setcounter{equation}{0}
%\setcounter{figure}{0}
%\setcounter{table}{0}
%\setcounter{section}{0}
%\setcounter{page}{1}
%\pagenumbering{roman}
%\makeatletter
%\renewcommand{\theequation}{A\arabic{equation}}
%\renewcommand{\thefigure}{A\arabic{figure}}
%\renewcommand{\thesection}{A\arabic{section}}
%\renewcommand{\bibnumfmt}[1]{[A#1]}
%\renewcommand{\citenumfont}[1]{A#1}
%%%%%%%%%%%%%%%%%%%%%%%%%%%%%%%%%%%%%%%%%%%%%%

%\section{Appendix}
%\subsection{}
%\section{Appendix}

\end{document}